\documentclass[12pt]{amsart}

\usepackage{amssymb}
\usepackage{amsmath}
\usepackage{amsthm}
\usepackage{amsfonts}

\usepackage{a4wide}
\usepackage{longtable}
\usepackage{multirow}
\usepackage{setspace}
\usepackage{stackrel}

\newtheorem{theorem}{Theorem}[section]
\newtheorem{proposition}{Proposition}[section]
\newtheorem{corollary}{Corollary}[section]

\newtheorem{conjecture}{Conjecture}[section]

\theoremstyle{definition}
\newtheorem{definition}{Definition}[section]
\newtheorem{remark}{Remark}[section]

\numberwithin{equation}{section}

\usepackage[all]{xy}

\usepackage{tabularx}

\begin{document}

\title{$3$-Principalization over $S_3$-fields}

\author{S. Aouissi}
\author{M. Talbi}
\author{D. C. Mayer}
\author{M. C. Ismaili}

\thanks{Research of the third author supported by the Austrian Science Fund (FWF): projects J0497-PHY and P26008-N25}

\subjclass[2010]{Primary 11R37, 11R29, 11R32, 11R20, 11R16, secondary 20D15, 20E22, 20F05, 20--04}
\keywords{Maximal unramified pro-$3$-extension, capitulation, Galois action, pure metacyclic fields of degree $6$, pure cubic fields;
finite \(3\)-groups, descendant trees, presentations.}

\date{Saturday, 06 March 2021}

\begin{abstract}
Let $p\equiv 1\,(\mathrm{mod}\,9)$ be a prime number
and $\zeta_3$ be a primitive cube root of unity.
Then $\mathrm{k}=\mathbb{Q}(\sqrt[3]{p},\zeta_3)$ is a pure metacyclic field with group
$\mathrm{Gal}(\mathrm{k}/\mathbb{Q})\simeq S_3$.
In the case that $\mathrm{k}$ possesses a $3$-class group $C_{\mathrm{k},3}$ of type $(9,3)$,
the capitulation of $3$-ideal classes of $\mathrm{k}$ in its unramified cyclic cubic extensions is determined,
and conclusions concerning the maximal unramified pro-$3$-extension $\mathrm{k}_3^{(\infty)}$ of $\mathrm{k}$ are drawn.
\end{abstract}

\maketitle


\section{Introduction}
\label{s:Introduction}

\noindent
Let $\mathrm{k}$ be a number field, and $L$ be an unramified abelian extension of $\mathrm{k}$.
We say that an ideal $\mathcal{I}$ of $\mathrm{k}$ or its class capitulates in $L$
if $\mathcal{I}$ becomes principal by extending it to an ideal of $L$.

\paragraph{}
Generally, it is difficult to determine the set of ideals of $\mathrm{k}$ which capitulate in $L$.
Among the first mathematicians who were concerned with this kind of problem is Kronecker,
who in 1882 studied the capitulation problem for imaginary quadratic fields.
   In 1902, Hilbert \cite{Hilbert} conjectured that
\textit{\lq\lq any ideal of a number field $\mathrm{k}$ capitulates in the Hilbert class field $\mathrm{k}^{(1)}$ of $\mathrm{k}$\rq\rq}.
This conjecture was demonstrated in 1930 by Furtw\"angler \cite{Furtwangler} after Artin had reduced it to a problem of group theory.
   Hilbert's famous Theorem 94 states that
\textit{\lq\lq if $L/\mathrm{k}$ is an unramified cyclic extension of prime degree $p$, then there exists an ideal of $\mathrm{k}$ of order $p$ which capitulates in $L$\rq\rq}.
The number of ideal classes of $\mathrm{k}$ that capitulate in such an extension $L$ is equal to $[L:\mathrm{k}][ E_{\mathrm{k}}:\mathcal{N}_{L/\mathrm{k}}(E_{L})]$,
where $E_{\mathrm{k}}$ and $E_{L}$ are respectively the groups of the units of $\mathrm{k}$ and $L$,
and $\mathcal{N}_{L/\mathrm{k}}$ is the relative norm of the extension $L/\mathrm{k}$.
   Around 1971, F. Terada, in his famous Tannaka-Terada theorem \cite{Tannaka-Terada} asserted that
\textit{\lq\lq if $\mathrm{k}/\mathrm{k}_0$ is a cyclic extension,
then all the ambiguous classes of $\mathrm{k}$, relative to $\mathrm{k}_0$, capitulate in the genus field of $\mathrm{k}$\rq\rq}.
   A generalization of Hilbert's Theorem 94 has been conjectured by Miyake \cite{Miyake} as follows:
\textit{\lq\lq the degree of an unramified abelian extension $L/\mathrm{k}$ divides the number of classes of $\mathrm{k}$ that capitulate in $L$\rq\rq}.
This conjecture was proved by Suzuki \cite{Suzuki1} in 1991,
who also succeeded, after 7 years of research, in proving a theorem \cite{Suzuki2} which generalizes all previous results as follows:
\textit{\lq\lq if $\mathrm{k}/\mathrm{k}_0$ is a cyclic extension, and $L$ is an unramified extension of $\mathrm{k}$ abelian over $\mathrm{k}_0$,
the degree $[L:\mathrm{k}]$ divides the number of classes of $\mathrm{k}$ invariant under $\mathrm{Gal}(\mathrm{k}/\mathrm{k}_0)$\rq\rq}.

\paragraph{}
Several mathematicians were interested in the study of the problem of capitulation for particular cases of number fields.
We cite, for example,
the study of capitulation in the intermediate extensions of $\mathrm{k}^{(1)}/\mathrm{k}$.
When the $p$-class group of $\mathrm{k}$ is of type $(p,p)$, this study was made by S. M. Chang and R. Foote.
For other studies we cite the works
\cite{Ayadi}, \cite{Azizi}, \cite{Is}, \cite{Kisilevesky}.
Here we contribute to this study with results on
the problem of capitulation of the $3$-ideal classes of the field $\mathrm{k}=\mathbb{Q}(\sqrt[3]{p},\zeta_3)$
in the intermediate extensions of $\mathrm{k}_3^{(1)}/\mathrm{k}$,
where $p$ is a prime such that $p\equiv 1\,(\mathrm{mod}\,9)$
and $\mathrm{k}_3^{(1)}$ denotes the \textit{Hilbert} $3$-class field of $\mathrm{k}$.

\paragraph{}
When the $3$-class group $C_{\mathrm{k},3}$ is of type $(9,3)$,
the extension $\mathrm{k}_3^{(1)}/\mathrm{k}$ admits eight intermediate fields.
Figure \ref{FF1} schematizes the situation.
 

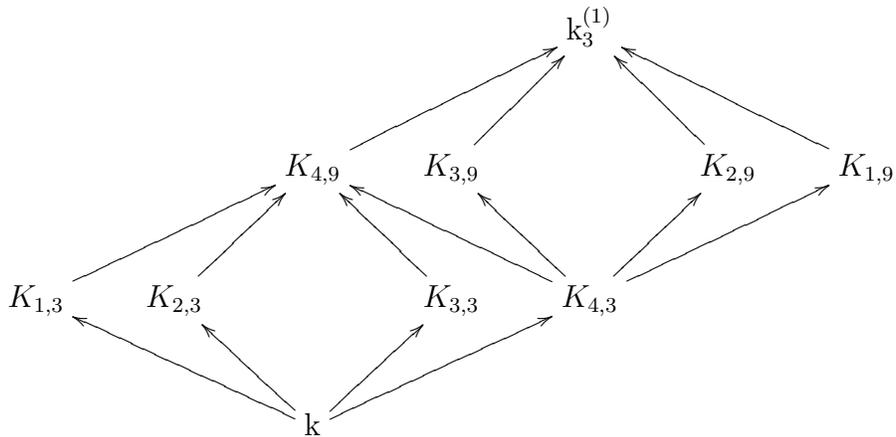
\begin{figure}[ht]
\caption{The unramified cubic and nonic sub-extensions of $ \mathrm{k}_3^{(1)}/\mathrm{k} $}
\label{FF1}
\[
\xymatrix@R=1,1cm{  
       {}&{}&{}&{} &{}&{}\\
       {}&{}&{}&{} & \mathrm{k_3^{(1)}}\ar@{<-}[rrd] \ar@{<-}[rd] \ar@{<-}[ld] \ar@{<-}[lld]&{}&{}\\       
       {}&{}& K_{4,9}\ar@{<-}[rrd] \ar@{<-}[rd] \ar@{<-}[ld] \ar@{<-}[lld] & K_{3,9} \ar@{<-}[rd] &{}& K_{2,9} \ar@{<-}[ld] & K_{1,9} \ar@{<-}[lld]\\
       K_{1,3} \ar@{<-}[rrd]& K_{2,3} \ar@{<-}[rd]&{}& K_{3,3} \ar@{<-}[ld] & K_{4,3} \ar@{<-}[lld]&{} \\
       {}&{}&  \mathrm{k} &{}&{}\\
       {}& {}& & {}& {}\\
  }
\]
\end{figure}

In section \ref{s:Capit}, we propose to determine the family $(K_{i,j})$ of all intermediate sub-fields of $\mathrm{k}\subseteq K_{i,j}\subseteq\mathrm{k}_3^{(1)}$,
where $1\leq i\leq 4$ and $j\in\lbrace 3,9\rbrace$, the object of Theorem \ref{thm:interm},
and we then study, in Theorem \ref{Capitulation93} of section \ref{s:Capit93}, the capitulation of $3$-ideal classes of $\mathrm{k}$ in the fields $K_{i,j}$.
Section \ref{s:Tower} is devoted to the identification of the
maximal unramified pro-$3$-extension $\mathrm{k}_3^{(\infty)}$ of $\mathrm{k}=\mathbb{Q}(\sqrt[3]{p},\zeta_3)$,
and to the proof that the dominating proportion (at least $94\%$) of the fields $\mathrm{k}$ has
a metabelian $3$-class field tower $\mathrm{k}_3^{(\infty)}$ with exactly two stages.

All our theoretical results are based on exhaustive computer results.
The usual notation is as follows:
    
\begin{itemize}
 \item The letter $p$ designates a prime number congruent to $1$ modulo $3$;
 \item $\Gamma=\mathbb{Q}(\sqrt[3]{d})$: a pure cubic field, where $d\ge 2$ is a cube-free integer;
 \item $\mathrm{k}_0=\mathbb{Q}(\zeta_3)$: the cyclotomic field, where $\zeta_3=e^{2i\pi/3}$ ;
 \item $\mathrm{k}=\Gamma(\zeta_3)$: the normal closure of $\Gamma$;
 \item $\Gamma^{\prime}$ and $\Gamma^{\prime\prime}$: the two conjugate cubic fields of $\Gamma$, contained in $\mathrm{k}$;
 \item $u=[E_{\mathrm{k}}:E_0]$: the index of the sub-group $E_0$ generated by the units of intermediate fields of the extension $\mathrm{k}/\mathbb{Q}$
 in the group of units $E_{\mathrm{k}}$ of $\mathrm{k}$;
 \item $\langle\tau\rangle=\operatorname{Gal}\left(\mathrm{k}/\Gamma\right)$, $\tau^2=id$, $\tau(\zeta_3)=\zeta_3^2$ and $\tau(\sqrt[3]{d})=\sqrt[3]{d}$;
 \item $\langle\sigma\rangle=\operatorname{Gal}\left(\mathrm{k}/\mathrm{k}_0\right)$, $\sigma^3=id$, $\sigma(\zeta_3)=\zeta_3$ and $\sigma(\sqrt[3]{d})=\zeta_3\sqrt[3]{d}$;
 \item For an algebraic number field $L$:
 \begin{itemize}
 \item $C_{L,3}$: the $3$-class group of $L$;
 \item $L_3^{(1)}$: the Hilbert $3$-class field of $L$;
 \item $\lbrack\mathcal{I}\rbrack$: the class of a fractional ideal $\mathcal{I}$ in the class group of $L$.
 \end{itemize}
\end{itemize}
 
\section{Unramified cubic and nonic sub-extensions of  $ \mathrm{k}_3^{(1)}/\mathrm{k}$}
\label{s:Capit}
\subsection{Preliminaries}
\label{s:Intro}
\paragraph{}
In his thesis \cite{Is}, M. C. Ismaili established that
the $3$-class group $C_{\mathrm{k},3}$ of $\mathrm{k}=\mathbb{Q}(\sqrt[3]{d},\zeta_3)$ is of type $(3,3)$
if and only if $3$ divides exactly the class number of $\Gamma$ and $u=3$,
where $u$ is the units index defined in the notations,
and he determined all the integers $d$ which satisfy this property by distinguishing three types of fields $\mathrm{k}$, namely type I, II, and III. 

\paragraph{}
Here, we are interested in the case where the $3$-class group $C_{\mathrm{k},3}$ is of type $(3,9)$.
Assuming that the number of classes of $\Gamma$ is divisible exactly by $9$,
Theorem \ref{39} describes the structure of $C_{\mathrm{k},3}$ as follows:

\begin{theorem}
\label{39}
Let $\Gamma$ be a pure cubic field,
\(\mathrm{k}\) be its normal closure, $C_{\mathrm{k},3}$ (resp. $C_{\Gamma,3}$) the $3$-class group of $\mathrm{k}$ (resp. $\Gamma$),
and \(u\) be the units index defined in the notations, then:
$$C_{\mathrm{k},3}\simeq\mathbb{Z}/9\mathbb{Z}\times\mathbb{Z}/3\mathbb{Z}
\quad\Longleftrightarrow\quad\lbrack C_{\Gamma,3}\simeq\mathbb{Z}/9\mathbb{Z}\quad \text{ and } \quad u=1\rbrack.$$
\end{theorem}

\begin{proof}
See \cite[Lemma 2.5, p. 4]{IJNT2}.
\end{proof}
 
\paragraph{}
Further, we have classified all the integers $d$ for which the $3$-class group $C_{\mathrm{k},3}$ is of type $(3,9)$ in the following Theorem:

\begin{theorem}
\label{T39}
Let $\Gamma=\mathbb{Q}(\sqrt[3]{d})$ be a pure cubic field,
where $d\ge 2$ is a cube free integer,
and let $\mathrm{k}=\mathbb{Q}(\sqrt[3]{d},\zeta_3)$ be its normal closure.
Denote by $u$ the index of the subgroup generated by the units of the intermediate fields of the extension $\mathrm{k}/\mathbb{Q}$
in the unit group of $\mathrm{k}$.
\begin{itemize}
\item[1)]
If the field $\mathrm{k}$ has a $3$-class group of type $(9,3)$,
then $d=p^{e}$, where $p$ is a prime congruent to $1\,(\mathrm{mod}\,9)$ and $e=1$ or $2$.
\item[2)]
Conversely, if $p$ is a prime congruent to $1\,(\mathrm{mod}\,9)$, and if
$9$ divides exactly the class number of $\Gamma=\mathbb{Q}(\sqrt[3]{p})$ and $u=1$,
then the $3$-class group of $\mathrm{k}$ is of type $(3,9)$.
\end{itemize}
\end{theorem}

\begin{proof}
See \cite[Theorem 1.1, p. 2]{IJNT2}.
\end{proof}
 
\paragraph{}
This being the case, let $ \mathrm{k} $ be the special field $\mathbb{Q}(\sqrt[3]{p},\zeta_3)$, with $p\equiv 1\,(\mathrm{mod}\,9)$.
Calegari and Emerton \cite[Lemma 5.11]{Cal-Emer} proved that the rank of the $3$-class group of this field $\mathrm{k}$
is equal to 2 if $9$ divides the class number of $\mathbb{Q}(\sqrt[3]{p})$.
The converse of the Calegari-Emerton result is shown by Frank Gerth III in \cite[Thm. 1, p. 471]{Ge2005}. Assume that $ 9 $ divides exactly the class number of $\Gamma=\mathbb{Q}(\sqrt[3]{p})$, where $p\equiv 1\pmod 9$, then $C_{\mathrm{k},3}$ is of type $(9,3)$ if and only if $u=1$. The generators of $C_{\mathrm{k},3}$ when $C_{\mathrm{k},3}$ is of  type $(9,3)$ are determined as follows:
 \begin{theorem}
\label{prop:carrr39}
Let $\mathrm{k}=\mathbb{Q}(\sqrt[3]{p},\zeta_3)$, where $p$ is a prime such that $p\equiv 1\,(\mathrm{mod}\,9)$, $\Gamma=\mathbb{Q}(\sqrt[3]{p})$, $\mathrm{k}_0=\mathbb{Q}(\zeta_3)$,  $\mathrm{k}_3^{(1)}$  the $3$-Hilbert class field of  $\mathrm{k}$,  $C_{\mathrm{k},3}$ the $3$-class group of $\mathrm{k}$, we note also $C_{\mathrm{k},3}^{+}=\{ \mathcal{C} \in C_{\mathrm{k},3} \, \vert\,
\mathcal{C}^{\tau}=\mathcal{C}  \}$, and $C_{\mathrm{k},3}^{-}=\{ \mathcal{C} \in C_{\mathrm{k},3} \, \vert\,
\mathcal{C}^{\tau}=\mathcal{C}^{-1} \}$, where $\langle\tau\rangle=\mathrm{Gal}(\mathrm{k}/\Gamma)$.
\\
 Suppose that
$C_{\mathrm{k},3}$ is of type $ (9,3)$.  Let $\langle \mathcal{A} \rangle=C_{\mathrm{k},3}^{+}$,
where $\mathcal{A} \in C_{\mathrm{k},3}$ such that $\mathcal{A}^{9}=1$
 and $\mathcal{A}^{3} \neq 1$,  \(C_{\mathrm{k},3}^{-}=\langle\mathcal{B}\rangle\),  where $ \mathcal{B} \in C_{\mathrm{k},3}$ such that $\mathcal{B}^{3}=1$
 and $\mathcal{B} \neq 1$, $C_{\mathrm{k},3}^{(\sigma)}$ is 
 the ambiguous ideal class group of
  $\mathrm{k}\vert\mathrm{k}_{0}$, 
 and   \(C_{\mathrm{k},3}^{1-\sigma}=\{ \mathcal{A}^{1-\sigma}\, \vert\,
\mathcal{A} \in C_{\mathrm{k},3} \}\) is the principal genus  of $C_{\mathrm{k},3}$, where  $\langle\sigma\rangle=\mathrm{Gal}(\mathrm{k}/\mathrm{k}_0)$. Then:
 \begin{enumerate}
 \item  $C_{\mathrm{k},3}= C_{\mathrm{k},3}^{+} \times C_{\mathrm{k},3}^{-} =\langle \mathcal{A}, \mathcal{B}\rangle$.
\item The ambiguous class group \(C_{\mathrm{k},3}^{(\sigma)}\) is a sub-group of \(C_{\mathrm{k},3}^{+}\) of order $3$, with $\mathcal{A} \not\in C_{\mathrm{k},3}^{(\sigma)} $, and we have:
$$C_{\mathrm{k},3}^{(\sigma)}=\langle\mathcal{A}^{3}\rangle=\langle\mathcal{B}^{1-\sigma}\rangle.$$
 \item \(C_{\mathrm{k},3}^{-}=\langle(\mathcal{A}^{2})^{\sigma-1}\rangle\).
  \item The principal genus
\(C_{\mathrm{k},3}^{1-\sigma}= C_{\mathrm{k},3}^{(\sigma)} \times C_{\mathrm{k},3}^{-}
=\langle \mathcal{A}^3, \mathcal{B} \rangle \) is an elementary bi-cyclic $3$-group of type \((3,3)\),
  \end{enumerate}
\end{theorem}
\begin{proof}
See \cite[Proposition 3.4, pp. 9-10]{boultim}.
\end{proof}
 \begin{theorem} \label{thm:main} 
Let $\mathrm{k}=\mathbb{Q}(\sqrt[3]{ p}, \zeta_3 )$, where $p$ is a prime such that $p\equiv 1\,(\mathrm{mod}\,9)$. The prime $ 3$  decomposes in $ \emph{k} $ under form $3\mathcal{O}_{\mathrm{k}}=
\mathcal{P}^{2}\mathcal{Q}^{2}\mathcal{R}^{2}$, where $\mathcal{P},
\,\,\mathcal{Q}$ and $\mathcal{R}$ are prime ideals of
$\mathrm{k}$. Put \(h=\frac{h_{\mathrm{k}}}{27}\),
where \(h_{\mathrm{k}}\) is the class number of \(\mathrm{k}\).
Assume $9$ divides exactly the class number of $\mathbb{Q}(\sqrt[3]{ p})$ and $u=1$. If $3$ is not cubic residue modulo $p$, then:
\begin{enumerate}
\item the class \([\mathcal{R}^{h}]\) generates $C_{\mathrm{k},3}^{+}$;
\item  \(C_{\mathrm{k},3}\)  is generated by classes \([\mathcal{R}^{h}]\) and
\([\mathcal{R}^{h}][\mathcal{P}^{h}]^2\), and we have:
\[C_{\mathrm{k},3}=\langle[\mathcal{R}^{h}]\rangle\times\langle[\mathcal{R}^{h}][\mathcal{P}^{h}]^2\rangle.\]
\end{enumerate}
\end{theorem}
\begin{proof}
See \cite[Theorem 3.5, pp. 10-11]{boultim}.
\end{proof}





\subsection{Unramified sub-extensions of  $ \mathrm{k}_3^{(1)}/\mathrm{k}$  }
\label{s:Subfields}

Let $\mathrm{k}=\mathbb{Q}(\sqrt[3]{p},\zeta_3)$, where $p$ is a prime such that $p\equiv 1\,(\mathrm{mod}\,9)$, $\mathrm{k}_3^{(1)}$  the $3$-Hilbert class field of $\mathrm{k}_3^{(0)}=\mathrm{k}$, $\mathrm{k}_3^{(2)}$ the $3$-Hilbert class field of $\mathrm{k}_3^{(1)}$, and $G=\mathrm{Gal}(\mathrm{k}^{(2)}_3/\mathrm{k})$. Let $C_{\mathrm{k},3}$ be $3$-ideal class group of  $\mathrm{k}$, then by class field theory, 
$\mathrm{Gal}(\mathrm{k}^{(1)}_3/\mathrm{k}) \simeq C_{\mathrm{k},3}$.

 Assume that $ C_{\mathrm{k},3}$ is of type $(9,3)$. In his paper \cite{93Dan}, Daniel C. Mayer gives the presentations of the metabelian $3$-groups  $G=\langle x,y\rangle$, with two generators satisfying $x^9 \in G^\prime$ and $y^3 \in G^\prime$, where $G/G^\prime$ is of type $(9,3)$. If we note by $(H_{i,j})_{i}$ the family of all normal intermediate groups $G^\prime\subseteq H_{i,j} \subseteq G,$ with $1\leq i\leq 4, \, 
j  \in \lbrace 3, 9 \rbrace $, the $3$-group $G$ has four maximal normal subgroups of index $ 9 $ as follows:
\begin{eqnarray*}
H_{1,9}  =  \langle y,G^\prime\rangle, \
H_{2,9}  =  \langle x^3y,G^\prime\rangle, \
H_{3,9}  = \langle x^3y^{-1},G^\prime\rangle, \
H_{4,9}  =  \langle x^3,G^\prime\rangle, 
\end{eqnarray*}
and four maximal normal subgroups of index $ 3 $ as follows:
\begin{eqnarray*}
H_{1,3} = \langle x,G^\prime\rangle, \
H_{2,3} = \langle xy,G^\prime\rangle, \
H_{3,3} = \langle xy^{-1},G^\prime\rangle, \
H_{4,3} = \langle x^3,y,G^\prime\rangle.
\end{eqnarray*}
It should be noted that
$H_{4,3}= \prod H_{i,9}$,  the quotient group  $H_{4,3}/G^\prime=\langle x^3,y\rangle$ is bi-cyclic of type $(3,3)$,  and
$H_{1,9}=\bigcap H_{i,3}=G^3G^\prime$ coincides with the Frattini subgroup $ \Phi(G) $ of $G$. However, the group $H_{i,9}$ is only contained in $H_{4,3}$, for $1\leq i\leq 3$. These groups are shown in figure \ref{FF2}.

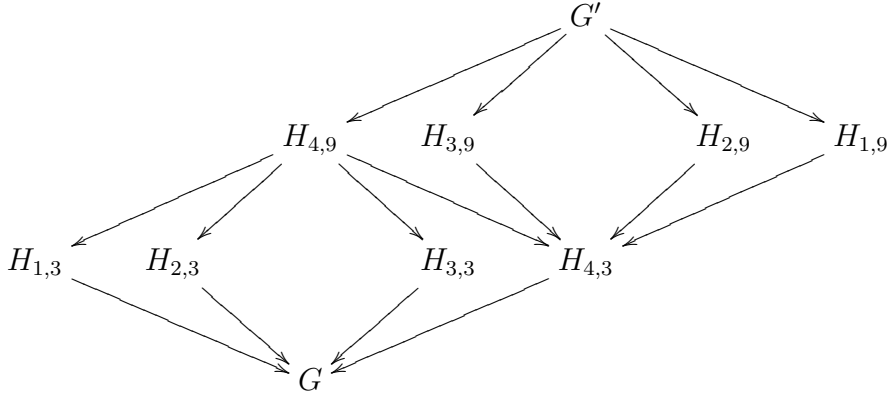
\begin{figure}
\caption{The group $G$ with $G/G^\prime$ of type $(9,3)$}
\label{FF2}
\[
\xymatrix@R=1cm{  
       {}&{}&{}&{} &{}&{}\\
       {}&{}&{}&{} & G^\prime \ar@{->}[rrd] \ar@{->}[rd] \ar@{->}[ld] \ar@{->}[lld]&{}&{}\\       
       {}&{}& H_{4,9}\ar@{->}[rrd] \ar@{->}[rd] \ar@{->}[ld] \ar@{->}[lld] & H_{3,9} \ar@{->}[rd] &{}& H_{2,9} \ar@{->}[ld] & H_{1,9} \ar@{->}[lld]\\
       H_{1,3} \ar@{->}[rrd]&  H_{2,3} \ar@{->}[rd]&{}&  H_{3,3} \ar@{->}[ld] &  H_{4,3} \ar@{->}[lld]&{} \\
       {}&{}&  G &{}&{}\\
       {}& {}& & {}& {}\\
  } \]
  \end{figure}


The aim of the following Theorem is to determine, via the Galois correspondence of $ \mathrm{k}_3^{(1)}/\mathrm{k}$, the family $(K_{i,j})$ of all fields $\mathrm{k}\subseteq K_{i,j}\subseteq \mathrm{k}_3^{(1)}$, where $1\leq i\leq 4, \, 
j \in \lbrace 3, 9 \rbrace$, satisfying $H_{i,j}= \mathrm{Gal}(\mathrm{k}_3^{(2)}/K_{i,j})$, 
 $H_{i,j}/G^\prime\simeq \mathcal{N}_{K_{i,j}/\mathrm{k}}(C_{K_{i,j},3})$, and $H_{i,j}/H_{i,j}^\prime\simeq \mathrm{Gal}((K_{i,j})_3^{(1)}/K_{i,j}) \simeq C_{K_{i,j},3}$.
 
\begin{theorem} \label{thm:interm}
 Let $\mathrm{k}=\mathbb{Q}(\sqrt[3]{p},\zeta_3)$, where $p$ is a prime such that $p\equiv 1\,(\mathrm{mod}\,9)$, $\mathrm{k}_3^{(1)}$  the 3-Hilbert class field of $\mathrm{k}$, and $C_{\mathrm{k},3}$ the $3$-ideal class group of $\mathrm{k}$. \\
 Assume
$C_{\mathrm{k},3}$ is of type $ (9,3)$. Let  $\langle \mathcal{A} \rangle=C_{\mathrm{k},3}^{+}$,
where $\mathcal{A} \in C_{\mathrm{k},3}$ such $\mathcal{A}^{9}=1$
 and $\mathcal{A}^{3} \neq 1$,  \(C_{\mathrm{k},3}^{-}=\langle\mathcal{B}\rangle\), where $ \mathcal{B} \in C_{\mathrm{k},3}$ such that $\mathcal{B}^{3}=1$
 and $\mathcal{B} \neq 1$, and let $C_{\mathrm{k},3}^{(\sigma)}$  the ambiguous ideal class group of
  $\mathrm{k}\vert\mathrm{k}_{0}$ . Then, the  extension
 \(\mathrm{k}_{3}^{(1)}/\mathrm{k}\) admits eight intermediate extensions as follows:
\begin{itemize}
 \item[1)] Four unramified cyclic extensions of degree $ 3 $ denoted $K_{i,3}, \, 1\leq i \leq 4$, given by: 
 \begin{itemize}
 \item The field
\(K_{1,3}\) corresponds by class field theory to
  $C_{\mathrm{k},3}^{+}=\langle \mathcal{A} \rangle $,
\item  The field \(K_{2,3}\) corresponds to
 $\langle \mathcal{A}\mathcal{B} \rangle=\langle \mathcal{A}^{\sigma} \rangle$,
  \item  The field \(K_{3,3}\) corresponds to  $\langle \mathcal{A}\mathcal{B}^2 \rangle=\langle \mathcal{A}^{\sigma^2} \rangle$,
,
\item  The field \(K_{4,3}\) corresponds  to
the principal genus
  $C_{\mathrm{k},3}^{1-\sigma}=\langle \mathcal{A}^3,\mathcal{B} \rangle$.\\
 Furthermore, 
\(K_{4,3}=\mathrm{k}(\sqrt[3]{\pi_1\pi_2^2})=\left(\mathrm{k}/\mathrm{k}_{0}\right)^{*}=\mathrm{k}\Gamma^{*}=
\mathrm{k}\left(\Gamma^{\sigma}\right)^{*}=\mathrm{k}\left(\Gamma^{\sigma^{2}}\right)^{*}\),\\ 
where $\left(\mathrm{k}/\mathrm{k}_{0}\right)^{*}$ is the relative genus field of $\mathrm{k}/\mathrm{k}_{0}$, $\langle\sigma\rangle= \mathrm{Gal}(\mathrm{k}/\mathrm{k}_{0})$, 
$F^{*}$ for a number field $F$ is the absolute genus field of $F$, and $\pi_1, \pi_2$ are primes in $\mathrm{k}_0$ such that $p=\pi_1\pi_2$.
\end{itemize}
\item[2)] Three unramified cyclic extensions of degree $ 9 $ denoted   $K_{i,9}, \, 2\leq i \leq 4$, given by:
\begin{itemize}
\item The field \(K_{2,9}\)
 corresponds  by class field theory to the sub-group \(\langle \mathcal{A}^3\mathcal{B} \rangle\), 
\item The field
$K_{3,9}$ corresponds to
 the  sub-group  \(\langle \mathcal{A}^3\mathcal{B}^2 \rangle \), 
\item The field $K_{1,9}$ corresponds
to the sub-group 
\(C_{\mathrm{k},3}^{-}=\langle\mathcal{B} \rangle\). \\ 
Furthermore,  $K_{1,9}=\mathrm{k}\cdot
{\Gamma_{3}}^{(1)}=\mathrm{k}\cdot
\left(\Gamma^{\sigma}\right)_{3}^{(1)}=\mathrm{k}\cdot
\left(\Gamma^{\sigma^{2}}\right)_{3}^{(1)}$, \\
where ${F_{3}}^{(1)}$ for a number field  $F$ is the  $3$-Hilbert class field of $F$.

\end{itemize}

\item[3)] One bi-cyclic bi-cubic extension of degree $9$, denoted $K_{4,9}$, and given by
 \(K_{4,9}=K_{i,3}\cdot K_{j,3},\, i\neq j  \), which corresponds by class field theory to the ambiguous ideal class group $C_{\mathrm{k},3}^{(\sigma)}=\langle \mathcal{A}^{3}\rangle$   of the extension
  $\mathrm{k}/\mathrm{k}_{0}$.
\end{itemize}
\end{theorem}

\begin{proof}
Assume 
$C_{\mathrm{k},3}$ is of type $ (3,9)$.  Let $\langle \mathcal{A} \rangle=C_{\mathrm{k},3}^{+}$,
where $\mathcal{A} \in C_{\mathrm{k},3}$ such that $\mathcal{A}^{9}=1$
 and $\mathcal{A}^{3} \neq 1$, and let \(C_{\mathrm{k},3}^{-}=\langle\mathcal{B}\rangle\),  where $ \mathcal{B} \in C_{\mathrm{k},3}$ such that $\mathcal{B}^{3}=1$
 and $\mathcal{B} \neq 1$.
 
The results of Theorem \ref{thm:interm} follow immediately, according to the class field theory, from the fact that the 3-ideal class group  $C_{\mathrm{k},3}=\langle
\mathcal{A},\mathcal{B}\rangle$ admits
\begin{itemize}
\item Four sub-groups  $H_{i,3}$ cubic of order $9$, where $1 \leq i \leq 4 $,  ordered as follows:
\begin{itemize}
\item three of these subgroups are cyclic of order $ 9 $, given by:
  \begin{itemize}
  \item[$\bullet$] $H_{1,3}=C_{\mathrm{k},3}^{+}=\langle \mathcal{A}\rangle =
\{ \mathcal{C} \in C_{\mathrm{k},3} \mid \mathcal{C}^{\tau }=\mathcal{C} \} $,
  \item[$\bullet$]  $ H_{2,3}= \langle  \mathcal{A} \mathcal{B}\rangle=\langle
\mathcal{A}^{\sigma}\rangle = 
\{ \mathcal{C} \in C_{\mathrm{k},3} \mid \mathcal{C}^{\tau \sigma }=\mathcal{C} \} $,  
 \item[$\bullet$] $ H_{3,3}= \langle \mathcal{A} \mathcal{B}^{-1} \rangle= \langle \mathcal{A} \mathcal{B}^2 \rangle
 =\langle
 \mathcal{A}^{\sigma^2}\rangle =
\{ \mathcal{C} \in C_{\mathrm{k},3} \mid \mathcal{C}^{\tau \sigma^2 }=\mathcal{C} \} $, 

\end{itemize}

 Then, we have:
      $$C_{\mathrm{k},3}/  H_{1,3} = C_{\mathrm{k},3}/ C_{\mathrm{k},3}^{+} =
     C_{\mathrm{k},3}/ \langle \mathcal{A}\rangle \simeq
      \mathbb{Z}/3\mathbb{Z},$$
      $$C_{\mathrm{k},3}/  H_{2,3} = 
     C_{\mathrm{k},3}/ \langle \mathcal{A}^{\sigma}\rangle \simeq
      \mathbb{Z}/3\mathbb{Z},$$
     $$C_{\mathrm{k},3}/  H_{3,3} = 
     C_{\mathrm{k},3}/ \langle \mathcal{A}^{\sigma^2}\rangle \simeq
      \mathbb{Z}/3\mathbb{Z},$$
      then $ \mathrm{Gal}\left( K_{1,3} / \mathrm{k} \right) \simeq  
      \mathrm{Gal}\left( K_{2,3} / \mathrm{k} \right) \simeq
      \mathrm{Gal}\left( K_{3,3} / \mathrm{k} \right) \simeq \mathbb{Z}/3\mathbb{Z} $, 
       which means that  $K_{1,3}$, $K_{2,3}$, $K_{3,3}$ are unramified cyclic extensions of degree $ 3 $ on $\mathrm{k}$
      corresponds respectively to the subgroups $H_{1,3}=\langle \mathcal{A}\rangle$, $H_{2,3}=\langle \mathcal{A}^{\sigma}\rangle$, $H_{3,3}=\langle \mathcal{A}^{\sigma^2}\rangle$  of $C_{\mathrm{k},3}$.
  \item the fourth subgroup of $C_{\mathrm{k},3}$  is exactly the principal genus $ C_{\mathrm{k},3}^{1-\sigma}$, given by $$ H_{4,3}= C_{\mathrm{k},3}^{1-\sigma}=
C_{\mathrm{k},3}^{(\sigma)} \times C_{\mathrm{k},3}^{-}
 =\langle \mathcal{A}^{3},\mathcal{B}\rangle = \prod_{i=1}^{4}  H_{i,9},$$
 then
  $$C_{\mathrm{k},3}/  H_{4,3} = C_{\mathrm{k},3}/ C_{\mathrm{k},3}^{1-\sigma} \simeq
      \mathbb{Z}/3\mathbb{Z},$$
 and according to genus theory $$ C_{\mathrm{k},3}/ C_{\mathrm{k},3}^{1-\sigma} \simeq \mathrm{Gal}\left( \left(\mathrm{k}/\mathrm{k}_{0}\right)^{*} / \mathrm{k} \right),  $$
 which is exactly is the genus group, for more details see
 \cite[\S 2, page 85]{Ge1976}.\\
 Then, $\mathrm{Gal}\left( K_{4,3} / \mathrm{k} \right) \simeq  \mathrm{Gal}\left( \left(\mathrm{k}/\mathrm{k}_{0}\right)^{*} / \mathrm{k} \right) \simeq \mathbb{Z}/3\mathbb{Z} $, which means that $K_{4,3}=\left(\mathrm{k}/\mathrm{k}_{0}\right)^{*}$ is an unramified cyclic extension of degree $ 3 $ over $\mathrm{k}$ correspond  to the sub-group $H_{4,3}=\langle \mathcal{A}^{3}, \mathcal{B}\rangle$  of $C_{\mathrm{k},3}$. For more details, see Theorem \ref{prop:carrr39}.\\
  Furthermore, as the discriminant of the ring of integers of $\Gamma$ is divisible by a single prime number $ p $ such that $p \equiv 1 \pmod 3$, then by \cite[Corollary 2.1, p. 21]{Is}, we have:
   \begin{eqnarray*}
   \Gamma^{*} & = & M(p).\Gamma, \\ \left({\Gamma^{\sigma}}\right)^{*} & = & M(p).\Gamma^{\sigma}, \\
   \left({\Gamma^{\sigma^2}}\right)^{*} & = & M(p).\Gamma^{\sigma^2},
   \end{eqnarray*}
   
where $\Gamma^{*}$ (respectively $\left({\Gamma^{\sigma}}\right)^{*}$, $\left({\Gamma^{\sigma^2}}\right)^{*}$) is the absolute genus field of  $\Gamma$ (respectively $\Gamma^{\sigma}$, $\Gamma^{\sigma^2}$), and  $M(p)$  is the unique cubic sub-field $\mathbb{Q}(\zeta_{p})$ of degree $3$.
  
By switching to the composition we obtain
    $$\mathrm{k}.\Gamma^{*}=
    \mathrm{k}.\left({\Gamma^{\sigma}}\right)^{*}=\mathrm{k}.\left({\Gamma^{\sigma^2}}\right)^{*} = \mathrm{k}. M(p).$$
 The fact that $ p \equiv 1 \pmod  3$ imply by  \cite[Chap. 9, Sec. 1, Proposition 9.1.4, p.110]{IlRo} that  $p=\pi_1 \pi_2$ with $\pi_1^{\tau}=\pi_2$ and $\pi_1 \equiv \pi_2 \equiv 1 \pmod {3\mathcal{O}_{\mathrm{k}_0}}$, then by  \cite[\S\ 3, Lemma 3.2, p. 56]{Ge1975}, we have $\left(\mathrm{k}/\mathrm{k}_{0}\right)^{*} =\mathrm{k}(\sqrt[3]{\pi_1\pi_2^2})$.\\
 We conclude that $K_{4,3}=\left(\mathrm{k}/\mathrm{k}_{0}\right)^{*}
 =\mathrm{k}\Gamma^{*}=\mathrm{k}(\sqrt[3]{\pi_1\pi_2^2}).$ 
\end{itemize}
    \item Four cyclic cubic subgroups  $H_{i,9}$ of order $3$, where $1 \leq i \leq 4 $, ordered as follows:
        \begin{itemize}
        \item[$\bullet$]  $H_{1,9}=C_{\mathrm{k},3}^{-}=\langle \mathcal{B}\rangle=\{\mathcal{C} \in C_{\mathrm{k},3} \mid
\mathcal{C}^{\tau}=\mathcal{C}^{-1}\}$,
   \item[$\bullet$]  $H_{2,9}=  \langle \mathcal{A}^{3} \mathcal{B}\rangle$,
    \item[$\bullet$] 
   $H_{3,9} =\langle \mathcal{A}^{3}\mathcal{B}^{-1}\rangle=\langle \mathcal{A}^{3}\mathcal{B}^{2}\rangle,$
 \item[$\bullet$] $H_{4,9}=C_{\mathrm{k},3}^{(\sigma)}=\langle \mathcal{A}^{3}\rangle$ is the ambiguous ideal class group of
  $\mathrm{k}/\mathrm{k}_{0}$.
    \end{itemize}
    Furthermore, $H_{4,9}= \bigcap_{i=1}^{4} H_{i,3}$, and for each $1 \leq i \leq 3$ we have $H_{i,9}$ is contained only in $ H_{4,3}$.
    \paragraph{}
  On the one hand, we have
    $$C_{\mathrm{k},3}/H_{1,9} = C_{\mathrm{k},3}/ C_{\mathrm{k},3}^{(\sigma)} =C_{\mathrm{k},3}/ \langle \mathcal{A}^{3}\rangle\simeq
   \mathbb{Z}/3\mathbb{Z}\times \mathbb{Z}/3\mathbb{Z},$$
   then 
   $ \mathrm{Gal}\left( K_{4,9} / \mathrm{k} \right) \simeq \mathbb{Z}/3\mathbb{Z}\times \mathbb{Z}/3\mathbb{Z} $, which signifies that $K_{4,9}$ is an unramified
     bi-cyclic bi-cubic extension of $\mathrm{k}$ corresponding to the sub-group $H_{1,9}=C_{\mathrm{k},3}^{(\sigma)}$ of $C_{\mathrm{k},3}$.\\
   Also, we have
      $$C_{\mathrm{k},3}/  H_{1,9} = C_{\mathrm{k},3}/ C_{\mathrm{k},3}^{-} =
     C_{\mathrm{k},3}/ \langle \mathcal{B}\rangle \simeq
      \mathbb{Z}/9\mathbb{Z},$$
     $$C_{\mathrm{k},3}/  H_{2,9} = 
     C_{\mathrm{k},3}/ \langle \mathcal{A}^3\mathcal{B}\rangle \simeq
      \mathbb{Z}/9\mathbb{Z},$$
      $$C_{\mathrm{k},3}/  H_{3,9} = 
     C_{\mathrm{k},3}/ \langle \mathcal{A}^3\mathcal{B}^2\rangle \simeq
      \mathbb{Z}/9\mathbb{Z},$$
      then $ \mathrm{Gal}\left( K_{1,9} / \mathrm{k} \right) \simeq  
      \mathrm{Gal}\left( K_{2,9} / \mathrm{k} \right) \simeq
      \mathrm{Gal}\left( K_{3,9} / \mathrm{k} \right) \simeq \mathbb{Z}/9\mathbb{Z} $, 
       which means that  $K_{1,9}$, $K_{2,9}$, $K_{3,9}$ are unramified cyclic extensions of degree $9$ over $\mathrm{k}$
      corresponding respectively to the sub-groups $H_{1,9}=\langle \mathcal{B}\rangle$, $H_{2,9}=\langle \mathcal{A}^3\mathcal{B}\rangle$, $H_{3,9}=\langle \mathcal{A}^3\mathcal{B}^2\rangle$ of $C_{\mathrm{k},3}$.
      
       \paragraph{}
     On the other hand, according to \cite[\S\ 2, Lemme 2.1, p. 53]{Ge1975} we have
      $$ C_{\mathrm{k},3}/  H_{1,9} = C_{\mathrm{k},3}/ C_{\mathrm{k},3}^{-} \simeq
      C_{\mathrm{k},3}^{+},$$ 
        by \cite[\S\ 2, Lemma 2.2, p. 53]{Ge1975} we have
       $$C_{\mathrm{k},3}^{+} \simeq C_{\Gamma,3},$$ 
       and according to class field theory  
       $$ C_{\Gamma,3} \simeq
       \mathrm{Gal}\left({\Gamma_3}^{(1)} \vert \Gamma\right),$$ 
      then we obtain
       $$C_{\mathrm{k},3}/  H_{1,9}  \simeq
       \mathrm{Gal}\left({\Gamma_3}^{(1)} \vert \Gamma\right) \simeq
       \mathrm{Gal}\left(\mathrm{k}{\Gamma_3}^{(1)} \vert
       \mathrm{k}\right),$$

      and by class field theory, $\mathrm{k}{\Gamma_3}^{(1)}$ is an unramified cyclic extension of degree $9$ of $\mathrm{k}$
       corresponding to the sub-group $H_{1,9}=C_{\mathrm{k},3}^{-}$ of $C_{\mathrm{k},3}$. Then, $K_{1,9}=\mathrm{k}{\Gamma_3}^{(1)}$. \\
      Now we show that
       $\mathrm{k}{\Gamma_3}^{(1)}=\mathrm{k}\left(\Gamma^{\sigma}\right)_{3}^{(1)}=
       \mathrm{k}\left(\Gamma^{{\sigma^2}}\right)_{3}^{(1)}.$
       Suppose that ${\Gamma_3}^{(1)}\neq \left(\Gamma^{\sigma}\right)_{3}^{(1)}$,
     then ${\Gamma_3}^{(1)}\cdot\left(\Gamma^{\sigma}\right)_{3}^{(1)}
       $ is an unramified extension of ${\Gamma_3}^{(1)}$ different to
       ${\Gamma_3}^{(1)}$. This contradicts the fact that the tower of class fields of $ \Gamma $ stops at the first stage $\Gamma$ since the $3$-ideal class group $C_{\Gamma,3}$ of $\Gamma$ is cyclic.
\end{itemize}
\end{proof}

\section{Capitulation of $3$-ideal classes of $\mathbb{Q}(\sqrt[3]{p},\zeta_3)$ }
\label{s:Capit93}
Let
$\mathrm{k}=\mathbb{Q}(\sqrt[3]{p},\zeta_3)$, where $p$ is a prime number such that $p\equiv 1\,(\mathrm{mod}\,9)$, and
$C_{\mathrm{k},3}$ the $3$-class group of $\mathrm{k}$.
Suppose that $ C_{\mathrm{k},3}$ is of type $(9,3)$. Let $(K_{i,j})$ be the family of all intermediate sub-fields of $\mathrm{k} \subseteq K_{i,j}\subseteq \mathrm{k}_3^{(1)}$,
where $1\leq i\leq 4$ and $j \in \lbrace 3, 9 \rbrace$. 
We denote by $\ker\left(T_{K_{i,j}/\mathrm{k}}\right)$ the
kernel of the  homomorphism 
$T_{K_{i,j}/\mathrm{k}}:\,C_{\mathrm{k},3} \longrightarrow  C_{K_{i,j},3}$ induced by extension of ideals of $\mathrm{k}$ in $K_{i,j}$,
where $K_{i,j}$ is an unramified extension of $\mathrm{k}$ included in $\mathrm{k}_3^{(1)}$. We denote by \(\kappa\) the quartet of Taussky's conditions
\cite{Ta}.
 
\begin{definition}
Let $\mathcal{A}_{i,j}$ be a generator of the sub-group $H_{i,j}$ of $C_{\mathrm{k},3}$,  with $1\leq i\leq 4, \,  j  \in \lbrace 3, 9 \rbrace $ corresponding to the field  $K_{i,j}$. Let  $l_{i} \in \lbrace 0, 1, 2, 3, 4 \rbrace $ with $1\leq i\leq 4$.\\
We will say that the capitulation is of type $\lbrace  l_1, l_2, l_3, l_4 \rbrace $ to express the fact that when $l_{i}=n$ for one $n \in \lbrace  1, 2, 3, 4 \rbrace  $, then only the class $\mathcal{A}_{n,9}$ and its powers capitulate in  $K_{i,3}$. If all classes capitulate in $K_{i,3}$ then we put $l_{i}=0$.
\end{definition}

\paragraph{}
The main result of this paper is as follows:

\begin{theorem}
\label{Capitulation93}
Let $\mathrm{k}=\mathbb{Q}(\sqrt[3]{p},\zeta_3)$,
where $p$ is a prime number such that $p\equiv 1\,(\mathrm{mod}\,9)$,
and $C_{\mathrm{k},3}$ is the $3$-class group of $\mathrm{k}$.
Suppose that $C_{\mathrm{k},3}$ is of type $(9,3)$.
Put $\langle\mathcal{A}\rangle=C_{\mathrm{k},3}^{+}$, where $\mathcal{A}\in C_{\mathrm{k},3}$
such that $\mathcal{A}^{9}=1$ and $\mathcal{A}^{3}\neq 1$.
Then:
\begin{enumerate}
\item
\begin{enumerate}
\item
\(K_{1,3}^{\sigma}=K_{2,3}\), \(K_{2,3}^{\sigma}=K_{3,3}\) and \(K_{3,3}^{\sigma}=K_{1,3}\)
($\sigma$ permutes $K_{1,3}$, $K_{2,3}$ and $K_{3,3}$).
\item
\(K_{1,3}^{\tau}=K_{1,3}\), \(K_{2,3}^{\tau}=K_{3,3}\) and \(K_{3,3}^{\tau}=K_{2,3}\),
\item
\(K_{4,9}^{\tau}=K_{4,9}\), \(K_{2,9}^{\tau}=K_{3,9}\) and \(K_{3,9}^{\tau}=K_{2,9}\),
\end{enumerate}
for all continuations of the automorphisms $\sigma$ and $\tau$.
\item
The three classes $\mathcal{A}$, $\mathcal{A}^{\sigma}$ and $\mathcal{A}^{\sigma^2}$
do not capitulate in \(K_{i,3}\), for \(1\leq i\leq 4\).
\item 
Exactly the class \(\mathcal{A}^{3}\) and its powers capitulate in
\(K_{4,3}\), and $\mathrm{ker}\left(T_{K_{4,3}/\mathrm{k}}\right)=\langle\mathcal{A}^3\rangle$, where
$T_{K_{4,3}/\mathrm{k}}:\,C_{\mathrm{k},3}\longrightarrow C_{K_{4,3},3}$ is the homomorphism induced by extension of ideals from $\mathrm{k}$ to $K_{4,3}$.
\item
The fields \(K_{2,3}\) and \(K_{3,3}\) have the same order of the capitulation kernel.
\item
The three classes $\mathcal{A}$, $\mathcal{A}^{\sigma}$ and $\mathcal{A}^{\sigma^2}$ capitulate in \(K_{1,9}\).
\item
The fields \(K_{2,9}\) and \(K_{3,9}\) have the same order of the capitulation kernel.
\item
Possible types of capitulation in $K_{i,3}$, $1\leq i\leq 4$, are
$(4,4,4;4)$, $(1,2,3;4)$ and $(0,0,0;4)$.
Possible Taussky types in $K_{i,3}$, $1\leq i\leq 4$, are
$\mathrm{(AAA;A)}$ or $\mathrm{(BBB;A)}$.
\end{enumerate}
\end{theorem}

\begin{proof}
Let $C_{\mathrm{k},3}^{(\sigma)}$ be the $3$-ambigous class group of $\mathrm{k}/\mathrm{k}_0$ and \(C_{\mathrm{k},3}^{1-\sigma}=\{ \mathcal{A}^{1-\sigma}\, \vert\,
\mathcal{A} \in C_{\mathrm{k},3} \}\)  the principal genus of  $C_{\mathrm{k},3}$.\\
By Theorem \ref{prop:carrr39} we have
\(C_{\mathrm{k},3}^{(\sigma)}=\langle\mathcal{A}^{3}\rangle=\langle\mathcal{B}^{1-\sigma}\rangle\), and  
\(C_{\mathrm{k},3}^{1-\sigma}= C_{\mathrm{k},3}^{-}\times
C_{\mathrm{k},3}^{(\sigma)} = \langle \mathcal{B},\mathcal{A}^{3}\rangle \) is a $3$-group of  $C_{\mathrm{k},3}$ of type \((3,3)\), where $ \mathcal{B} \in C_{\mathrm{k},3}$ such that \(C_{\mathrm{k},3}^{-}=\langle\mathcal{B}\rangle =\langle(\mathcal{A}^{2})^{\sigma-1}\rangle\).
\begin{enumerate}
\item
We will agree that for all $i$, $1\leq i\leq 4$, $j=3$ or $9$, and for all $\omega \in \mathrm{Gal}\left( k \vert \mathbb{Q}\right)$, $H_{i,j}^{\omega}=\lbrace \mathcal{C}^{\omega} / \mathcal{C} \in H_{i,j}\rbrace.$
\begin{enumerate}
\item
According to Theorem \ref{thm:interm}, $H_{1,3}=C_{\mathrm{k},3}^{+}=\langle\mathcal{A}\rangle,
 H_{2,3}= \{\mathcal{C} \in C_{\mathrm{k},3} \vert
\mathcal{C}^{\tau\sigma}=\mathcal{C}\}=\langle\mathcal{A}^{\sigma}\rangle,$ and $ H_{3,3}= \{\mathcal{C} \in C_{\mathrm{k},3} \vert
\mathcal{C}^{\tau\sigma^{2}}=\mathcal{C}\}=\langle\mathcal{A}^{\sigma^2}\rangle.   $
Then, \(H_{1,3}^{\sigma}=H_{2,3}\),
\(H_{2,3}^{\sigma}=H_{3,3}\) and \(H_{3,3}^{\sigma}=H_{1,3}\)  ($\sigma$ permutes $H_{1,3}$, $H_{2,3}$ and $H_{3,3}$).

\item
As $H_{1,3}=C_{\mathrm{k},3}^{+}=\{ \mathcal{C} \in C_{\mathrm{k},3} \, \vert\,
\mathcal{C}^{\tau}=\mathcal{C}  \}$, then \(H_{1,3}^{\tau}=H_{1,3}\).
 We have 
$H_{2,3}^{\tau}= \langle(\mathcal{A}^{\sigma})^{\tau}\rangle=\langle\mathcal{A}^{\tau\sigma}\rangle $, and since $\mathcal{A}^{\tau\sigma}=\mathcal{A}^{\sigma^2\tau}=(\mathcal{A}^{\tau})^{\sigma^2}=\mathcal{A}^{\sigma^2} \in H_{3,3}$, then $H_{2,3}^{\tau}=H_{3,3}.$
$H_{3,3}^{\tau}= \langle(\mathcal{A}^{\sigma^2})^{\tau}\rangle=\langle\mathcal{A}^{\tau\sigma^2}\rangle=\langle\mathcal{A}^{\sigma\tau}\rangle=\langle(\mathcal{A}^{\tau})^{\sigma}\rangle=\langle\mathcal{A}^{\sigma}\rangle $, then $H_{3,3}^{\tau}=H_{2,3}$.

\item
We reason as in $(2)$. \(H_{1,9}^{\tau}=H_{1,9}\) because $H_{1,9}=C_{\mathrm{k},3}^{-}=\{ \mathcal{C} \in C_{\mathrm{k},3} \, \vert\,
\mathcal{C}^{\tau}=\mathcal{C}^{-1}\}$. We have $H_{2,9}=\langle \mathcal{A}^3\mathcal{B} \rangle $, and
$ H_{3,9}=\langle \mathcal{A}^3\mathcal{B}^{2} \rangle $, and since
  \(\mathcal{A}^{\tau}=\mathcal{A}\) and \(\mathcal{B}^{\tau}=\mathcal{B}^{-1}=\mathcal{B}^{2}\),
then \(H_{2,9}^{\tau}=H_{3,9}\) and \(H_{3,9}^{\tau}=H_{2,9}\). 

\end{enumerate}
The relations between the fields $K_{i,j}$ in (1) are nothing else than the translations of the corresponding relations for the sub-groups $H_{i,j}$ via class field theory.

\item
For each \(1\leq i\leq 4\), \(K_{i,3}\) is an unramified cyclic extension of degree $3$ over \(\mathrm{k}\). It is  clear that for each class \(\mathcal{C}  \in C_{\mathrm{k},3}\) we have
\(\mathcal{C}^{3}=(N_{K_{i,3}/ \mathrm{k}}\circ T_{K_{i,3}/\mathrm{k}})(\mathcal{C} )\).
If the class $\mathcal{C} $ capitulates in \(K_{i,3}\),
then  $ T_{K_{i,3}/\mathrm{k}}(\mathcal{C})=1$ and $\mathcal{C}^3=1$. 
We conclude that the ideal classes which capitulate in \(K_{i,3}\) are of order \(3\). Since the
classes \(\mathcal{A}, \, \mathcal{A}^{\sigma}
 \), and $\mathcal{A}^{\sigma^2}$ are of order
\(9\), then these classes cannot capitulate in \(K_{i,3}\).

\item
By Theorem \ref{thm:interm} we have \(K_{4,3}=\left(\mathrm{k}/\mathrm{k}_{0}\right)^{*}\) is the relative genus field of $\mathrm{k}/\mathrm{k}_{0}$,  and by Theorem \ref{prop:carrr39} we have
\(\langle\mathcal{A}^{3}\rangle=\langle\mathcal{A}^{3\sigma}\rangle=\langle\mathcal{A}^{3\sigma^2}\rangle=C_{\mathrm{k},3}^{(\sigma)}\). We conclude according to Tannaka-Terada theorem \cite{Tannaka-Terada}, that all ambiguous ideal classes of $C_{\mathrm{k},3}$ capitulate in the relative genus field $\left(\mathrm{k}/\mathrm{k}_{0}\right)^{*}$.  Thus, the class \(\mathcal{A}^{3} \) and its powers capitulate in
\(K_{4,3}\). 
We shall prove that the unique  classes which capitulate in \(K_{4,3}\) are only the ambiguous ideal classes. We have \(C_{\mathrm{k},3}^{-}=\langle\mathcal{B}\rangle =\langle(\mathcal{A}^{2})^{\sigma-1}\rangle\) and 
\(C_{\mathrm{k},3}^{(\sigma)}=\langle\mathcal{A}^{3}\rangle
=\langle\mathcal{B}^{1-\sigma}\rangle\). On the one hand we have
\begin{equation*}
 \mathcal{A}^{1+2\sigma} = \mathcal{A}^{\sigma(1-\sigma)} = ((\mathcal{A}^{-1})^{\sigma-1})^{\sigma} = ((\mathcal{A}^{2})^{\sigma-1})^{\sigma} = \mathcal{B}^{\sigma}
\end{equation*}
because $(\mathcal{A}^{3})^{\sigma-1}=1$. One the other hand, we have
\begin{equation*}
 \mathcal{B}^{1-\sigma} \mathcal{B}^{2} = \mathcal{B}^{3-\sigma} = \mathcal{B}^{-\sigma}
\end{equation*}
then
\begin{equation*}
 (\mathcal{B}^{1-\sigma} \mathcal{B}^{2})^{-1} = \mathcal{B}^{\sigma} 
\end{equation*}
So we get
\begin{equation*}
  \mathcal{B}^{\sigma} \in  \langle \mathcal{B}^{1-\sigma} \mathcal{B}^2 \rangle = \langle \mathcal{A}^{3} \mathcal{B}^2 \rangle
\end{equation*}
because $\langle\mathcal{A}^{3}\rangle
=\langle\mathcal{B}^{1-\sigma}\rangle$.
Since
$C_{\mathrm{k},3}= \langle \mathcal{A}, \mathcal{B} \rangle$ is of type $(9,3)$ , then a class $\chi \in C_{\mathrm{k},3}$ of order 3 capitulates in  
the cubic cyclic unramified  extension \(K_{4,3}/\mathrm{k}\), if and only if the $ \mathcal{B}$
capitulates in the extension \(K_{4,3}/\mathrm{k}\),
because a class $\chi $ of order 3 is in one of the subgroup
$\langle \mathcal{B} \rangle$, $\langle \mathcal{A}^3 \mathcal{B} \rangle$,  $\langle \mathcal{A}^3 \mathcal{B}^2 \rangle$ and that $\mathcal{B}^{\sigma} \in  \langle \mathcal{A}^{3} \mathcal{B}^2 \rangle$. \\
If $ \mathcal{B}$
capitulates in the extension \(K_{4,3}/\mathrm{k}\), then
$ \mathcal{B}^{\sigma}$
capitulates also in  \(K_{4,3}/\mathrm{k}\). Since $\mathcal{A}^{1+2\sigma}=\mathcal{B}^{\sigma}$, then
$\mathcal{A}^{1+2\sigma}$ capitulates also in  \(K_{4,3}/\mathrm{k}\),
so $T_{K_{4,3}/\mathrm{k}}\left( \mathcal{A}^{1+2\sigma}\right)=1$. Then
\begin{equation*}
\left( T_{K_{4,3}/\mathrm{k}}\left( \mathcal{A}^{\sigma}\right)\right)^2=
T_{K_{4,3}/\mathrm{k}}\left( \mathcal{A}^{-1}\right)= T_{K_{4,3}/\mathrm{k}}\left( \mathcal{A}^{2}\right)
\end{equation*}
because $T_{K_{4,3}/\mathrm{k}}\left( \mathcal{A}^{3}\right)=1$, so
\begin{equation*}
\left( T_{K_{4,3}/\mathrm{k}}\left( \mathcal{A}^{\sigma}\right)\right)^2= \left( T_{K_{4,3}/\mathrm{k}}\left( \mathcal{A}\right)\right)^2
\end{equation*}
and then
\begin{equation*}
 T_{K_{4,3}/\mathrm{k}}\left( \mathcal{A}^{\sigma}\right)=  T_{K_{4,3}/\mathrm{k}}\left( \mathcal{A}\right)
\end{equation*}
so we get $\left(\Gamma^\prime_{3}\right)^{(1)}=\Gamma_{3}^{(1)}$, where $\Gamma_{3}^{(1)}$ (resp. $\left(\Gamma^\prime_{3}\right)^{(1)}=\left(\Gamma^{\sigma}_{3}\right)^{(1)}$) is the $3$-Hilbert class field of $\Gamma$ (resp. $\Gamma^\prime$), which is a contradiction.\\
Thus  $\mathcal{B}$ does not capitulate in  \(K_{4,3}/\mathrm{k}\), and then only $\mathcal{A}^3$ and its powers capitulate in \(K_{4,3}/\mathrm{k}\).

\item
\(K_{2,3}\) and \(K_{3,3}\) have the same order of the capitulation kernel, because $K_{2,3}$ and $K_{3,3}$ are isomorphic by (1)(b).

\item
Let $\mathcal{I}$ be an ideal of $\Gamma$ whose class generates $C_{\Gamma,3}$.
We know then that we can take $\mathcal{A}=[T_{\mathrm{k}/\Gamma}(\mathcal{I})]$,
that the class of $\mathcal{I}^{\sigma}$ (in $\Gamma^{\sigma}=\Gamma^\prime$) generates $C_{\Gamma^\prime,3}$, 
that the class of $\mathcal{I}^{\sigma^2}$ (in $\Gamma^{\sigma^2}=\Gamma^{\prime\prime}$) generates $C_{\Gamma^{\prime\prime},3}$, and that 
$\mathcal{A}^{\sigma}=[T_{\mathrm{k}/\Gamma}(\mathcal{I})^{\sigma}]=[T_{\mathrm{k}/\Gamma^\prime}(\mathcal{I}^{\sigma})]$
and $\mathcal{A}^{\sigma^2}=[T_{\mathrm{k}/\Gamma}(\mathcal{I})^{\sigma^2}]=[T_{\mathrm{k}/\Gamma^{\prime\prime}}(\mathcal{I}^{\sigma^2})]$.

Since $\Gamma_{3}^{(1)}$ (resp. $\left(\Gamma^\prime_{3}\right)^{(1)}$ and $\left(\Gamma^{\prime\prime}_{3}\right)^{(1)}$)
is the $3$-Hilbert class field of $\Gamma$ (resp. $\Gamma^\prime$ and $\Gamma^{\prime\prime}$),
then $\mathcal{I}$ (resp. $\mathcal{I}^{\sigma}$ and $\mathcal{I}^{\sigma^2}$) becomes principal
in $\Gamma_{3}^{(1)}$ (resp. $\left(\Gamma^\prime_{3}\right)^{(1)}$ and $\left(\Gamma^{\prime\prime}_{3}\right)^{(1)}$).
Thus, when $\mathcal{I}$ (resp. $\mathcal{I}^{\sigma}$ and $\mathcal{I}^{\sigma^2}$) is considered as an ideal of
$\mathrm{k}.\Gamma_{3}^{(1)}$ (resp. $\mathrm{k}.\left(\Gamma^\prime_{3}\right)^{(1)}$ and $\mathrm{k}.\left(\Gamma^{\prime\prime}_{3}\right)^{(1)}$), $\mathcal{I}$ (resp. $\mathcal{I}^{\sigma}$ and $\mathcal{I}^{\sigma^2}$) becomes principal in
$\mathrm{k}.\Gamma_{3}^{(1)}$ (resp. $\mathrm{k}.\left(\Gamma^\prime_{3}\right)^{(1)}$ and $\mathrm{k}.\left(\Gamma^{\prime\prime}_{3}\right)^{(1)}$). 
So 
$\mathcal{A}$ (resp. $\mathcal{A}^{\sigma}$ and $\mathcal{A}^{\sigma^2}$)  capitulates in
$\mathrm{k}.\Gamma_{3}^{(1)}$ (resp. $\mathrm{k}.\left(\Gamma^\prime_{3}\right)^{(1)}$ and $\mathrm{k}.\left(\Gamma^{\prime\prime}_{3}\right)^{(1)}$).
Since
$\mathrm{k}.\Gamma_{3}^{(1)}=\mathrm{k}.\left(\Gamma^\prime_{3}\right)^{(1)}=\mathrm{k}.\left(\Gamma^{\prime\prime}_{3}\right)^{(1)}=K_{1,9}$, then 
the classes $\mathcal{A}, \,\mathcal{A}^{\sigma}$ and $\mathcal{A}^{\sigma^2}$ capitulate in \(K_{1,9}\).

\item
This assertion follows from the fact that $K_{2,9}$ and $K_{3,9}$ are isomorphic by (1)(c).

\item 
The possible types of capitulation  in $K_{i,3},\, 1\leq i\leq 4$ are $(4,4,4;4)$, $(1,2,3;4)$ and $(0,0,0;4)$,
and the possible Taussky types are (AAA;A) or (BBB;A). In fact, since exactly the class \(\mathcal{A}^{3} \) and its powers capitulate in
\(K_{4,3}\), then:
\begin{itemize}
\item[$(i)$] for each $j \in \lbrace 1,2,3 \rbrace$, if all ideal classes  $\mathcal{C}$ of $C_{\mathrm{k},3}$ of order  3  capitulate in $K_{j,3}/\mathrm{k}$, then  the type of capitulation is $(0,0,0;4)$, and the Taussky type is (AAA;A).
\item[$(ii)$] If exactly the class $\mathcal{C}$ of $C_{\mathrm{k},3}$  of order 3 capitulates in the extension $K_{1,3}/\mathrm{k}$, then exactly one class of $C_{\mathrm{k},3}$ of order 3 and its powers capitulate in the extensions $K_{2,3}/\mathrm{k}$  and $K_{3,3}/\mathrm{k}$, because by (1)(a) we have $K_{1,3}^{\sigma}=K_{2,3}$, $K_{2,3}^{\sigma}=K_{3,3}$, and $K_{3,3}^{\sigma}=K_{1,3}$. Then, there are two cases:
\begin{itemize}
\item If exactly the class $ \mathcal{A}^{3}$ capitulates in $K_{1,3}/\mathrm{k}$, then exactly the class $ (\mathcal{A}^{3})^{\sigma}$ capitulates in  $K_{2,3}/\mathrm{k}$ and exactly the class $ (\mathcal{A}^{3})^{\sigma^2}$ capitulates in  $K_{3,3}/\mathrm{k}$.
In this case, the possible type of capitulation is $(4,4,4;4)$,  and the Taussky type is (AAA;A).
\item We have \(C_{\mathrm{k},3}^{-}=\langle\mathcal{B}\rangle\), and
\(C_{\mathrm{k},3}^{(\sigma)}=\langle\mathcal{A}^{3}\rangle
=\langle\mathcal{B}^{1-\sigma}\rangle\) by Theorem \ref{prop:carrr39}. Then

\begin{eqnarray*}
 \mathcal{B}^{1-\sigma} \mathcal{B}^{2} & = & \mathcal{B}^{3-\sigma} \\
 & = & \mathcal{B}^{-\sigma},
\end{eqnarray*}

that is

\begin{eqnarray*}
 (\mathcal{B}^{1-\sigma} \mathcal{B}^{2})^{-1} & = & \mathcal{B}^{\sigma}.
\end{eqnarray*}

We get
\begin{equation*}
  \mathcal{B}^{\sigma} \in  \langle \mathcal{B}^{1-\sigma} \mathcal{B}^2 \rangle = \langle \mathcal{A}^{3} \mathcal{B}^2 \rangle,
\end{equation*}

because $\langle\mathcal{A}^{3}\rangle
=\langle\mathcal{B}^{1-\sigma}\rangle$.
\\

If exactly the class $ \mathcal{B}$ capitulates in $K_{1,3}/\mathrm{k}$, then exactly the class $ \mathcal{B}^{\sigma}$ capitulates in $K_{2,3}/\mathrm{k}$ and exactly the class $ \mathcal{B}^{\sigma^2}$ capitulates in $K_{3,3}/\mathrm{k}$. Then, exactly the class $ \mathcal{A}^{3}\mathcal{B}^{2}$ capitulates in  $K_{2,3}/\mathrm{k}$ and exactly the class $ \mathcal{A}^{3}\mathcal{B}$ capitulates in  $K_{3,3}/\mathrm{k}$.\\
 In this case, the possible type of capitulation is  $(1,2,3;4)$,  and the Taussky type is (BBB;A).
\end{itemize}

\end{itemize}

\end{enumerate}

\end{proof} 


\section{$3$-class field tower of $\mathbb{Q}(\sqrt[3]{p},\zeta_3)$ }
\label{s:Tower}

\noindent
Let \(\mathrm{k}=\mathbb{Q}(\sqrt[3]{p},\zeta_3)\) be the normal closure
of the pure cubic field \(\Gamma=\mathbb{Q}(\sqrt[3]{p})\)
with prime radicand \(p\equiv 1\,(\mathrm{mod}\,9)\)
of \textit{Dedekind's second species}.
Then \(\mathrm{k}\) is a pure metacyclic field with absolute group
\(\mathrm{Gal}(\mathrm{k}/\mathbb{Q})\simeq S_3\)
the symmetric group of order six.
Assume that \(\mathrm{k}\) possesses a \(3\)-class group
\(C_{\mathrm{k},3}\simeq C_9\times C_3\).
Consequently, the \(3\)-class group of \(\Gamma\) is
\(C_{\Gamma,3}\simeq C_9\),
according to Theorem
\ref{39},
and \(\Gamma\) is of \textit{principal factorization type} \(\alpha\),
in the sense of
\cite{AMITA2020}.

In Theorem
\ref{Capitulation93},
we investigated the principalization of \(\mathrm{k}\)
in its four unramified cyclic cubic extensions \(K_{1,3},\ldots,K_{4,3}\),
i.e., we determined the three possibilities for the kernels
\(\ker(T_{K_{i,3}/\mathrm{k}})\)
of the \textit{transfer homomorphisms}
\(T_{K_{i,3}/\mathrm{k}}:\,C_{\mathrm{k},3}\to C_{K_{i,3},3}\),
\(\mathfrak{a}\mathcal{P}_{\mathrm{k}}\mapsto(\mathfrak{a}\mathcal{O}_{K_{i,3}})\mathcal{P}_{K_{i,3}}\).

Our numerical results for the \(95\) relevant cases \(p<20\,000\) in Table
\ref{tbl:Experiments},
computed with the aid of Magma
\cite{MAGMA},
confirm the occurrence of precisely three situations
for the \textit{punctured capitulation type}
\(\varkappa(\mathrm{k})=(\ker(T_{K_{i,3}/\mathrm{k}}))_{1\le i\le 4}\),
which we want to dub with succinct names in Definition
\ref{dfn:CapitulationTypes}.
Recall that \(H_{4,9}=\cap_{j=1}^4\,H_{j,3}\) in Figure
\ref{FF2}
is the \textit{distinguished} subgroup of \(C_{\mathrm{k},3}\)
which is generated by third powers of \(3\)-ideal classes,
i.e., the Frattini subgroup.


\begin{definition}
\label{dfn:CapitulationTypes}
The punctured capitulation type,
with puncture at the fourth component,
for the subfield \(K_{4,3}\) associated with
the subgroup \(H_{4,3}=\prod_{j=1}^4\,H_{j,9}\) in Figure
\ref{FF2},
is called
\begin{enumerate}
\item
\textit{distinguished},
if \(\varkappa(\mathrm{k})=(H_{4,9},H_{4,9},H_{4,9};H_{4,9})\),
briefly \((444;4)\),
\item
\textit{harmonically balanced},
if \(\varkappa(\mathrm{k})=(H_{1,9},H_{2,9},H_{3,9};H_{4,9})\),
briefly \((123;4)\),
\item
\textit{total},
if \(\varkappa(\mathrm{k})=(H_{4,3},H_{4,3},H_{4,3};H_{4,9})\),
briefly \((000;4)\).
\end{enumerate}
\end{definition}


\noindent
For the actual numerical determination of the (punctured) capitulation type \(\varkappa\),
we introduce the concept of Artin pattern of \(\mathrm{k}\).

\begin{definition}
\label{dfn:ArtinPattern}
Let \(\tau(\mathrm{k})=\lbrack\mathrm{ATI}(C_{K_{i,3},3})\rbrack_{1\le i\le 4}\) be
the family of \textit{abelian type invariants} (ATI)
(i.e., \(3\)-\textit{primary type invariants})
of the \(3\)-class groups \(C_{K_{i,3},3}\)
of the four unramified cyclic cubic extensions of \(\mathrm{k}\).
Then \(\mathrm{AP}(\mathrm{k})=(\varkappa(\mathrm{k}),\tau(\mathrm{k}))\)
is called the \textit{Artin pattern} of \(\mathrm{k}\).
\end{definition}

It turns out that there is
a \textit{bijective correspondence} between \(\varkappa\) and \(\tau\)
for the distinguished and total capitulation,
whereas there are \textit{two variants} of harmonically balanced capitulation.

Anyway, it is never required to perform the difficult computation of
the capitulation type \(\varkappa\).
It is sufficient to determine the abelian type invariants \(\tau\),
which is computationally easier.
Theorem
\ref{thm:SimplifiedComputation}
is a consequence of our results in Table
\ref{tbl:Experiments}.


\begin{theorem}
\label{thm:SimplifiedComputation}
For a pure metacyclic field \(\mathrm{k}=\mathbb{Q}(\zeta_3,\sqrt[3]{p})\)
with prime radicand \(p\equiv 1\,(\mathrm{mod}\,9)\), bounded by \(p\le 20\,000\),
and \(3\)-class group \(C_{\mathrm{k},3}\simeq C_9\times C_3\),
the following statements determine \(\varkappa(\mathrm{k})\) by means of \(\tau(\mathrm{k})\):
\begin{enumerate}
\item
\(\varkappa(\mathrm{k})=(444;4)\) \(\Longleftrightarrow\) \(\tau(\mathrm{k})=\lbrack (9,3)^3;(9,3)\rbrack\).
\item
\(\varkappa(\mathrm{k})=(000;4)\) \(\Longleftrightarrow\) \(\tau(\mathrm{k})=\lbrack (9,3,3)^3;(3,3,3,3)\rbrack\).
\item
\(\varkappa(\mathrm{k})=(123;4)\) \(\Longleftrightarrow\)
\(\tau(\mathrm{k})=\begin{cases}
\text{either } \lbrack (27,3)^3;(9,3,3)\rbrack & \text{(\(1^{\text{st}}\) variant)} \\
\text{or }     \lbrack (27,3)^3;(9,9,3)\rbrack & \text{(\(2^{\text{nd}}\) variant)}.
\end{cases}\)
\end{enumerate}
\end{theorem}


\begin{conjecture}
\label{cnj:SimplifiedComputation}
Theorem
\ref{thm:SimplifiedComputation}
is true for any prime \(p\equiv 1\,(\mathrm{mod}\,9)\),
not necessarily bounded from above by \(20\,000\).
\end{conjecture}


\noindent
We are now in the position to employ the
\textit{strategy of pattern recognition via Artin transfers}
\cite{Ma2020}
in order to determine the \(3\)-class field tower \(\mathrm{k}_3^{(\infty)}\) of \(\mathrm{k}\)
by means of \(\mathrm{AP}(\mathrm{k})=(\varkappa(\mathrm{k}),\tau(\mathrm{k}))\).


\subsection{Relation rank and Galois action}
\label{ss:Constraints}

\noindent
Constraints arise from two issues,
bounds for the relation rank of the tower group \(G=\mathrm{Gal}(\mathrm{k}_3^{(\infty)}/\mathrm{k})\),
and the Galois action of \(\mathrm{Gal}(\mathrm{k}/\mathbb{Q})\) on \(C_{\mathrm{k},3}\simeq G/G^\prime\).
We denote by \(\langle o,i\rangle\) groups in the SmallGroups database of Magma
\cite{MAGMA}.


\begin{theorem}
\label{thm:RelationRankAndGaloisAction}
For any pure metacyclic field \(\mathrm{k}=\mathbb{Q}(\zeta_3,\sqrt[3]{d})\)
with cube free radicand \(d\ge 2\) and \(3\)-class rank \(\varrho=2\),
the group \(G=\mathrm{Gal}(\mathrm{k}_3^{(\infty)}/\mathrm{k})\) of the \(3\)-class field tower
must satisfy the following conditions.
\begin{enumerate}
\item
The relation rank \(d_2\) of \(G\) must be bounded by
\(2\le d_2\le 5\).
\item
The automorphism group \(\mathrm{Aut}(Q)\) of the Frattini quotient \(Q=G/\Phi(G)\)
must contain a subgroup isomorphic to
\(S_3=\langle 6,1\rangle\).
(This is true for any \(S_3\)-field \(\mathrm{k}\).)
\end{enumerate}
\end{theorem}

\begin{proof}
According to the Burnside basis theorem,
the generator rank \(d_1\) of \(G\) coincides with
the generator rank of the Frattini quotient \(Q=G/\Phi(G)=G/(G^\prime\cdot G^3)\),
resp. the derived quotient \(G/G^\prime\simeq C_{\mathrm{k},3}\),
that is the \(3\)-class rank \(\varrho\) of \(\mathrm{k}\).
\begin{enumerate}
\item
According to the Shafarevich Theorem
\cite[Thm. 5.1, p. 28]{Ma2015d},
the relation rank \(d_2\) of \(G\) is bounded by
\(d_1\le d_2\le d_1+r+\vartheta\),
where the torsion free unit rank \(r=r_1+r_2-1\) of the
totally complex field \(\mathrm{k}\) with signature \((r_1,r_2)=(0,3)\)
is \(r=2\),
and \(\vartheta=1\), since \(k\) contains the primitive third roots of unity.
Together with the generator rank \(d_1=\varrho=2\) this gives the bounds
\(2\le d_2\le 2+2+1=5\).
(For other complex, resp. real, \(S_3\)-fields \(\mathrm{k}\), the upper bound may be \(4\), resp. \(7\).)
\item
The absolute Galois group \(\mathrm{Gal}(\mathrm{k}/\mathbb{Q})\simeq S_3\) of \(\mathrm{k}\)
acts on the \(3\)-class group \(C_{\mathrm{k},3}\simeq G/G^\prime\)
and thus also on the Frattini quotient \(Q=G/\Phi(G)=G/(G^\prime\cdot G^3)\),
whence \(\mathrm{Aut}(Q)\) contains a subgroup isomorphic to \(S_3=\langle 6,1\rangle\).
\qedhere
\end{enumerate}
\end{proof}

\noindent
By the same proof as for item (2) of Theorem
\ref{thm:RelationRankAndGaloisAction},
with \(G/G^\prime\simeq C_{\mathrm{k},3}\)
replaced by \(G_n/G_n^\prime\simeq\)
\(\mathrm{Gal}(\mathrm{k}_3^{(n)}/\mathrm{k})/\mathrm{Gal}(\mathrm{k}_3^{(n)}/\mathrm{k}_3^{(1)})\simeq
\mathrm{Gal}(\mathrm{k}_3^{(1)}/\mathrm{k})\simeq C_{\mathrm{k},3}\)
we obtain:

\begin{corollary}
\label{cor:GaloisAction}
Let \(n\) be a positive integer,
and denote by \(G_n=\mathrm{Gal}(\mathrm{k}_3^{(n)}/\mathrm{k})\)
the Galois group of the \(n\)-th Hilbert \(3\)-class field \(\mathrm{k}_3^{(n)}\) of \(\mathrm{k}\).
The automorphism group \(\mathrm{Aut}(Q)\) of the Frattini quotient \(Q=G_n/\Phi(G_n)\)
must contain a subgroup isomorphic to
\(S_3=\langle 6,1\rangle\).
\end{corollary}


\noindent
Furthermore, it will also be required to exploit data concerning
the \textit{second layer} of unramified abelian
(three cyclic nonic and a single bicyclic bicubic) extensions.

\begin{definition}
\label{dfn:ArtinPattern2}
Let \(\varkappa_2(\mathrm{k})=(\ker(T_{K_{i,9}/\mathrm{k}}))_{1\le i\le 4}\) be
the \textit{punctured capitulation type},
and \(\tau_2(\mathrm{k})=\lbrack\mathrm{ATI}(C_{K_{i,9},3})\rbrack_{1\le i\le 4}\) be
the family of \textit{abelian type invariants}
of the \(3\)-class groups \(C_{K_{i,9},3}\)
of the four unramified abelian nonic extensions of \(\mathrm{k}\),
and \(\mathrm{AP}_2(\mathrm{k})=(\varkappa_2(\mathrm{k}),\tau_2(\mathrm{k}))\).
\end{definition}

\noindent
According to item (5) of Theorem
\ref{Capitulation93},
we know that \(\ker(T_{K_{1,9}/\mathrm{k}})=C_{\mathrm{k},3}\).


\subsection{Distinguished capitulation}
\label{ss:Distinguished}


\begin{proposition}
\label{prp:DistinguishedTower}
A power commutator presentation of the finite metabelian \(3\)-group \(\langle 81,4\rangle\)
with class \(2\) and coclass \(2\)
in terms of the commutator \(s_2=\lbrack y,x\rbrack\) is given by
\begin{equation}
\label{eqn:PCPresentationOrd81Id4}
\langle x,y,s_2\mid x^9=1,\ y^3=s_2\rangle 
\end{equation}
\end{proposition}

\begin{proof}
Presentations of groups in the SmallGroups database are implemented in Magma
\cite{MAGMA}.
\end{proof}


\begin{theorem}
\label{thm:DistinguishedTower}
For a pure metacyclic field \(\mathrm{k}=\mathbb{Q}(\zeta_3,\sqrt[3]{p})\)
with \(p\equiv 1\,(\mathrm{mod}\,9)\)
having distinguished capitulation \(\varkappa(\mathrm{k})=(444;4)\),
the Galois group \(G_2\) of the second Hilbert \(3\)-class field \(\mathrm{k}_3^{(2)}\)
is unambiguously given by
\(\mathrm{Gal}(\mathrm{k}_3^{(2)}/\mathrm{k})\simeq
\langle 81,4\rangle\) with \(\varkappa_2(\mathrm{k})=((C_{\mathrm{k},3})^3;C_{\mathrm{k},3})\) and \(\tau_2(\mathrm{k})=\lbrack (9)^3;(3,3)\rbrack\)
(see Figure
\ref{fig:C9xC3}).
The \(3\)-class field tower of \(\mathrm{k}\)
must stop at the second stage, that is,
\(\mathrm{k}_3^{(2)}=\mathrm{k}_3^{(\infty)}\) is the maximal unramified pro-\(3\)-extension of \(\mathrm{k}\).
\end{theorem}


\begin{proof} (Proof of Theorem \ref{thm:DistinguishedTower})
We use Theorem
\ref{thm:SimplifiedComputation}
in order to exploit the equivalence
\(\varkappa(\mathrm{k})=(444;4)\) \(\Longleftrightarrow\) \(\tau(\mathrm{k})=\lbrack(9,3)^3;(9,3)\rbrack\).
According to the \textit{theorem on the antitony}
\(\varkappa(P)\ge\varkappa(D)\) and \(\tau(P)\le\tau(D)\)
of the components of the Artin pattern \((\varkappa,\tau)\)
with respect to (parent, descendant)-pairs \((P,D)\),
where \(P\) is a quotient of \(D\),
the abelian type invariants \(\tau=\lbrack (9,3)^3;(9,3)\rbrack\),
which are common to the metabelian \(3\)-groups
\(\langle 81,4\rangle\) and \(\langle 243,22\rangle\),
cannot occur for any other finite \(3\)-group.
Any other finite \(3\)-group is descendant of the metabelian root \(R=\langle 81,3\rangle\)
with pc-presentation \(\langle x,y,s_2\mid x^9=1,\ y^3=1,\ s_2=\lbrack y,x\rbrack\rangle\)
and \(\tau(R)=\lbrack (9,3)^3,(3,3,3)\rbrack\).
Consequently, at least one component of \(\tau(D)\) will always be of rank three,
for any descendant \(D\) of the root \(R\).
Furthermore, this argument also shows that there cannot be a non-metabelian \(3\)-group \(G\)
with second derived quotient \(G/G^{\prime\prime}\) isomorphic to
either \(\langle 81,4\rangle\) or \(\langle 243,22\rangle\),
since \(G\) would necessarily be required to have \(\tau(G)=\lbrack (9,3)^3);(9,3)\rbrack\),
which is not compatible with being a descendant of \(R\).
According to the Artin reciprocity law of class field theory,
the \(3\)-class field tower of \(\mathrm{k}\) must therefore have precise length \(\ell_3(\mathrm{k})=2\).
Finally, both candidates for \(G_2=G\) satisfy the inequalities
\(2\le d_2\le 5\) for the relation rank in Theorem
\ref{thm:RelationRankAndGaloisAction}.
Indeed, \(\langle 81,4\rangle\) has \(d_2=3\),
and \(\langle 243,22\rangle\) has even the minimal value \(d_2=2\).
However, for \(G=\langle 81,4\rangle\),
the automorphism group \(\mathrm{Aut}(Q)\)
of the Frattini quotient \(Q=G/\Phi(G)\)
contains a subgroup isomorphic to \(S_3=\langle 6,1\rangle\),
whereas for \(G=\langle 243,22\rangle\),
the corresponding \(\mathrm{Aut}(Q)\)
contains a subgroup \(C_2=\langle 2,1\rangle\) only.
Thus, \(G=\langle 81,4\rangle\) remains as unique candidate.
See Figure
\ref{fig:C9xC3}.
\end{proof}

\noindent
We point out that the preceding proof is not really dependent on Theorem
\ref{thm:SimplifiedComputation}.
A search for \(\varkappa(\mathrm{k})=(444;4)\) in the SmallGroups database
yields \(\langle 81,4\rangle\), \(\langle 243,22\rangle\)
and descendants of \(\langle 729,10\rangle\), \(\langle 729,12\rangle\).
However, the latter two roots (and thus all of their descendants)
do not have the required action by \(S_3\).


\subsection{Harmonically balanced capitulation}
\label{ss:HarmonicallyBalanced}

\noindent
We have seen that harmonically balanced capitulation \(\varkappa=(123;4)\)
occurs in \textit{two variants} with distinct fourth components
\((9,3,3)\), resp. \((9,9,3)\), in the abelian type invariants \(\tau\).
It turns out that the \textit{first variant} leads to
\textit{sporadic} groups outside of coclass trees,
and the \textit{second variant} is connected with
\textit{periodic} groups on coclass trees.


\begin{proposition}
\label{prp:FirstVariantTower}
A power commutator presentation of the finite metabelian \(3\)-group
\(\langle 729,i\rangle\)
of class \(3\)
in terms of the commutators \(s_2=\lbrack y,x\rbrack\), \(s_3=\lbrack s_2,x\rbrack\), \(t_3=\lbrack s_2,y\rbrack\) is given by
\begin{equation}
\label{eqn:PCPresentationOrd729Id17And20}
\begin{cases}
\langle x,y,s_2,s_3,t_3\mid x^9=t_3,\ y^3=s_3\rangle   & \text{if } i=17, \\
\langle x,y,s_2,s_3,t_3\mid x^9=t_3^2,\ y^3=s_3\rangle & \text{if } i=20.
\end{cases}
\end{equation}
\end{proposition}

\begin{proof}
Presentations of groups in the SmallGroups database are implemented in Magma
\cite{MAGMA}.
The groups are sporadic of coclass \(3\).
\end{proof}


\begin{theorem}
\label{thm:FirstVariantTower}
For a pure metacyclic field \(\mathrm{k}=\mathbb{Q}(\zeta_3,\sqrt[3]{p})\)
with \(p\equiv 1\,(\mathrm{mod}\,9)\)
having harmonically balanced capitulation \(\varkappa(\mathrm{k})=(123;4)\)
and \textbf{first} variant of \(\tau=\lbrack (27,3)^3; (9,3,3)\rbrack\),
the \textbf{sporadic} Galois group \(G_2\) of \(\mathrm{k}_3^{(2)}\), second Hilbert \(3\)-class field,
is given by
\(\mathrm{Gal}(\mathrm{k}_3^{(2)}/\mathrm{k})\simeq\)
\begin{equation}
\label{eqn:FirstVariantTower}
\begin{cases}
\langle 729,\ell\rangle      & \text{ if } \varkappa_2(\mathrm{k})=((C_{\mathrm{k},3})^3;H_{4,3}),\ \tau_2(\mathrm{k})=\lbrack (9,3)^3;(9,3,3)\rbrack, \\
\langle 2187,m\rangle        & \text{ if } \varkappa_2(\mathrm{k})=((C_{\mathrm{k},3})^3;H_{4,3}),\ \tau_2(\mathrm{k})=\lbrack (9,9)^3;(9,9,3)\rbrack,
\end{cases}
\end{equation}
where \(\ell\in\lbrace 17,20\rbrace\), \(m\in\lbrace 177,178,187,188\rbrace\) (see Figures
\ref{fig:Ord729Id17}
and
\ref{fig:Ord729Id20}).
\end{theorem}

\begin{proof}
(Proof of Theorem
\ref{thm:FirstVariantTower})
All vertices of the entire descendant trees of the roots
\(\langle 729,i\rangle\) with \(i\in\lbrace 17,20\rbrace\)
share the required Artin pattern \((\varkappa,\tau)\)
with harmonically balanced capitulation \(\varkappa=(123;4)\)
and the first variant of \(\tau=\lbrack (27,3)^3;(9,3,3)\rbrack\).
Since the trees are isomorphic as \textit{structured} graphs,
we focus on \(\langle 729,17\rangle\),
which gives rise to a \textit{finite} \lq\lq mainline\rq\rq,
standing out through an action by the direct product \(S_3\times C_2\simeq\langle 12,4\rangle\).
The metabelian vertices of this finite mainline are 
\(\langle 729,17\rangle\), \(\langle 2187,178\rangle\), \(\langle 6561,1733\rangle\),
and \(\langle 6561,1733\rangle-\#1;2\).
The other two immediate descendants of the root \(\langle 729,17\rangle\) are
\(\langle 2187,177\rangle\) with action by \(S_3\)
and \(\langle 2187,179\rangle\) with action by \(C_2\) only.
There are exactly two further candidates for \(G_2\) with action by \(S_3\),
namely the metabelian groups \(\langle 6561,1731\rangle\) and \(\langle 6561,1733\rangle-\#1;3\).
However,
\(\langle 6561,n\rangle\) with \(n\in\lbrace 1731,1733\rbrace\) and
\(\langle 6561,1733\rangle-\#1;s\) with \(s\in\lbrace 2,3\rbrace\)
share the forbidden second layer
\(\varkappa_2=(H_{1,3},H_{2,3},H_{3,3};H_{4,3})\), \(\tau_2=\lbrack (27,9)^3;(9,9,3)\rbrack\).
See Figures
\ref{fig:Ord729Id17}
and
\ref{fig:Ord729Id20}.
\end{proof}


\begin{proposition}
\label{prp:SecondVariantTower}
A power commutator presentation of the finite metabelian \(3\)-group
\(\langle 2187,i\rangle\)
in terms of the commutators \(s_2=\lbrack y,x\rbrack\), \(s_3=\lbrack s_2,x\rbrack\), \(s_4=\lbrack s_3,x\rbrack\), \(t_3=\lbrack s_2,y\rbrack\) is given by
\begin{equation}
\label{eqn:PCPresentationOrd2187Id180And190}
\begin{cases}
\langle x,y,s_2,s_3,s_4,t_3\mid x^9=t_3,\ y^3=s_3^2,\ s_2^3=s_4^2\rangle   & \text{if } i=180, \\
\langle x,y,s_2,s_3,s_4,t_3\mid x^9=t_3^2,\ y^3=s_3^2,\ s_2^3=s_4^2\rangle & \text{if } i=190.
\end{cases}
\end{equation}
\end{proposition}

\begin{proof}
Presentations of groups in the SmallGroups database are implemented in Magma
\cite{MAGMA}.
The groups are periodic of class \(4\) and coclass \(3\).
\end{proof}


\begin{theorem}
\label{thm:SecondVariantTower}
For a pure metacyclic field \(\mathrm{k}=\mathbb{Q}(\zeta_3,\sqrt[3]{p})\)
with \(p\equiv 1\,(\mathrm{mod}\,9)\)
having harmonically balanced capitulation \(\varkappa(\mathrm{k})=(123;4)\)
and \textbf{second} variant of \(\tau=\lbrack (27,3)^3; (9,9,3)\rbrack\),
the \textbf{periodic} Galois group \(G_2\) of the second Hilbert \(3\)-class field \(\mathrm{k}_3^{(2)}\)
is given by 
\(\mathrm{Gal}(\mathrm{k}_3^{(2)}/\mathrm{k})\simeq\)
\begin{equation}
\label{eqn:SecondVariantTower}
\begin{cases}
\langle 2187,m\rangle       & \text{ if } \varkappa_2(\mathrm{k})=((C_{\mathrm{k},3})^3;H_{4,3}),\ \tau_2(\mathrm{k})=\lbrack (9,9)^3;(9,9,3)\rbrack, \\
\langle 6561,n\rangle       & \text{ if } \varkappa_2(\mathrm{k})=((C_{\mathrm{k},3})^3;H_{4,3}),\ \tau_2(\mathrm{k})=\lbrack (27,9)^3;(9,9,9)\rbrack,
\end{cases}
\end{equation}
where \(m\in\lbrace 180,190\rbrace\) and \(n\in\lbrace 1737,1738,1739,1775,1776,1777\rbrace\) (see Figures
\ref{fig:Ord729Id18} and \ref{fig:Ord729Id21}).
\end{theorem}

\begin{proof}
(Proof of Theorem
\ref{thm:SecondVariantTower})
The required Artin pattern \((\varkappa,\tau)\)
with harmonically balanced capitulation \(\varkappa=(123;4)\)
and second variant of \(\tau=\lbrack (27,3)^3;(9,9,3)\rbrack\)
cannot occur for descendants of the roots \(\langle 729,i\rangle\) with \(i\in\lbrace 17,20\rbrace\),
because on the entire descendant trees of these sporadic roots
\(\tau=\lbrack (27,3)^3;(9,3,3)\rbrack\) remains \textit{stable}.

The only possibility are vertices of the coclass trees
with roots \(\langle 729,i\rangle\) for \(i\in\lbrace 18,21\rbrace\).
Since the trees are isomorphic as \textit{structured} graphs,
we focus on \(\langle 729,21\rangle\),
which has three immediate descendants,
\(\langle 2187,190\rangle\) with \(\varkappa=(123;4)\), \(\tau=\lbrack (27,3)^3;(9,9,3)\rbrack\),
the mainline group \(\langle 2187,191\rangle\) with host type \(\varkappa=(123;0)\) like the parent \(\langle 729,21\rangle\),
and \(\langle 2187,192\rangle\) with inadequate \(\varkappa=(123;2)\).
Due to the \textit{antitony principle} for the components of the Artin pattern \((\varkappa,\tau)\),
all descendants of \(\langle 2187,191\rangle\) can be eliminated,
because they have \(\tau\ge\lbrack (27,3)^3;(27,9,3)\rbrack\).
The group \(\langle 2187,190\rangle\) has the required action by \(S_3=\langle 6,1\rangle\),
and this is also true for three of its immediate descendants
\(\langle 6561,n\rangle\) with \(1775\le n\le 1777\)
but not for \(n=1778\) with action by \(C_3=\langle 3,1\rangle\) only.
Each of the three former has an immediate descendant
\(\langle 6561,n\rangle-\#1;1\) with \(1775\le n\le 1777\) and action by \(S_3\).
The other descendant \(\langle 6561,n\rangle-\#1;2\) has action by \(C_3\),
and three further descendants \(\langle 6561,n\rangle-\#1;1-\#1;i\) with \(1\le i\le 3\)
have only an action by \(C_2=\langle 2,1\rangle\).
Further suitable candidates for \(G_2\) are impossible.
Finally,  the groups \(\langle 6561,n\rangle-\#1;1\) with \(n\in\lbrace 1775,1776,1777\rbrace\)
are discouraged by a wrong transfer kernel in the second layer with
\(\varkappa_2=(H_{1,3},H_{2,3},H_{3,3};H_{4,3})\), \(\tau_2=\lbrack (27,27)^3;(9,9,9)\rbrack\).
See Figures
\ref{fig:Ord729Id18}
and
\ref{fig:Ord729Id21}.
\end{proof}


\noindent
The proofs of the subsequent corollaries are based on the following fact.
All the candidates for \(G_2\) 
in Theorem \ref{thm:FirstVariantTower} and Theorem \ref{thm:SecondVariantTower}
satisfy the inequalities
\(2\le d_2\le 5\) for the \textit{relation rank} in Theorem
\ref{thm:RelationRankAndGaloisAction},
since they even satisfy the more severe estimates \(3\le d_2\le 4\).
So there is no reason which precludes a metabelian tower with length \(\ell_3(\mathrm{k})=2\).


\begin{corollary}
\label{cor:FirstVariantTower}
For the fields \(\mathrm{k}\) with harmonically balanced capitulation \(\varkappa=(123;4)\)
and \textbf{first} variant of \(\tau=\lbrack (27,3)^3; (9,3,3)\rbrack\)
(Theorem \ref{thm:FirstVariantTower})
the \(3\)-class tower of \(\mathrm{k}\) must stop at the second stage, that is,
\(\mathrm{k}_3^{(2)}=\mathrm{k}_3^{(\infty)}\) is the maximal unramified pro-\(3\)-extension of \(\mathrm{k}\).
\end{corollary}

\begin{proof}
(Proof of Corollary
\ref{cor:FirstVariantTower})
The groups \(G_2\) in Theorem
\ref{thm:FirstVariantTower}
are not second derived quotients \(G/G^{\prime\prime}\) of non-metabelian \(3\)-groups \(G\).
\end{proof}

\begin{remark}
\label{rmk:FirstVariantTower}
We emphasize that the strict limitation \(\ell_3(\mathrm{k})=2\)
for the length of the \(3\)-class field tower of \(\mathrm{k}\) in Corollary
\ref{cor:FirstVariantTower}
is only due to item (5) of Theorem
\ref{Capitulation93},
i.e. the requirement \(\ker(T_{K_{1,9}/\mathrm{k}})=C_{\mathrm{k},3}\)
for the second layer.

Although we also had \(\ell_3(\mathrm{k})=2\)
if \(G_2=\langle 6561,n\rangle\), \(n\in\lbrace 1731,1769\rbrace\), were admissible,
\(\ell_3(\mathrm{k})=3\) would be enabled for \(r\in\lbrace 1733,1771\rbrace\), and
\(\langle 6561,r\rangle\simeq G/G^{\prime\prime}\) with \(G=\langle 6561,r\rangle-\#1;4\),
resp.
\(\langle 6561,r\rangle-\#1;s\simeq G/G^{\prime\prime}\)
with \(G=\langle 2187,m\rangle-\#2;1-\#1;s\), \(m\in\lbrace 178,188\rbrace\), \(s\in\lbrace 2,3\rbrace\).
The groups \(G\) are non-metabelian with action by \(S_3\),
in contrast to \(\langle 6561,r\rangle\) with \(r\in\lbrace 1735,1773\rbrace\).
On the other hand, it is known that the descendant trees of the roots
\(\langle 729,i\rangle\) with \(i\in\lbrace 17,20\rbrace\)
contain vertices with unbounded derived length,
whence any finite value \(\ell_3(\mathrm{k})\ge 4\) would also be possible.
(Note that the tree continues at the non-metabelian vertex
\(\langle 2187,m\rangle-\#2;1-\#1;2\) with \(m=178\), resp. \(m=188\).)
\end{remark}


\begin{corollary}
\label{cor:SecondVariantTower}
For the fields \(\mathrm{k}\) with harmonically balanced capitulation \(\varkappa=(123;4)\)
and \textbf{second} variant of \(\tau=\lbrack (27,3)^3; (9,9,3)\rbrack\)
(Theorem \ref{thm:SecondVariantTower})
the \(3\)-class tower of \(\mathrm{k}\) must stop at the second stage, that is,
\(\mathrm{k}_3^{(2)}=\mathrm{k}_3^{(\infty)}\) is the maximal unramified pro-\(3\)-extension of \(\mathrm{k}\).
\end{corollary}

\begin{proof}
(Proof of Corollary
\ref{cor:SecondVariantTower})
According to the \textit{antitony principle},
there cannot exist non-metabelian \(3\)-groups \(G\)
whose metabelianization \(G/G^{\prime\prime}\)
is isomorphic to one of the \(14\) candidates
for \(\mathrm{Gal}(\mathrm{k}_3^{(2)}/\mathrm{k})\) in Theorem
\ref{thm:SecondVariantTower}.
Therefore \(\mathrm{k}_3^{(2)}=\mathrm{k}_3^{(\infty)}\).
\end{proof}


\subsection{Total capitulation}
\label{ss:Total}

\noindent
Due to the wealth of metabelian groups \(M\) of low orders \(\#M\le 3^8\)
in the descendant tree of the root \(\langle 729,9\rangle\),
we restrict ourselves to \textit{immediate} descendants of the root with action by \(S_3\).


\begin{proposition}
\label{prp:TotalTower}
A power commutator presentation of the finite metabelian \(3\)-group \(\langle 729,9\rangle\)
in terms of the commutators \(s_2=\lbrack y,x\rbrack\), \(s_3=\lbrack s_2,x\rbrack\), \(t_3=\lbrack s_2,y\rbrack\) is given by
\begin{equation}
\label{eqn:PCPresentationOrd729Id9}
\langle x,y,s_2,s_3,t_3\mid x^9=1,\ y^3=1\rangle.
\end{equation}
\end{proposition}

\begin{proof}
Presentations of groups in the SmallGroups database are implemented in Magma
\cite{MAGMA}.
The group is periodic of class \(3\) and coclass \(3\).
\end{proof}

\begin{theorem}
\label{thm:TotalTower}
For a pure metacyclic field \(\mathrm{k}=\mathbb{Q}(\zeta_3,\sqrt[3]{p})\)
with \(p\equiv 1\,(\mathrm{mod}\,9)\)
having total capitulation \(\varkappa(\mathrm{k})=(000;4)\)
and abelian type invariants \(\tau=\lbrack (9,3,3)^3; (3,3,3,3)\rbrack\),
the smallest possible Galois groups \(G_2\) of the second Hilbert \(3\)-class field \(\mathrm{k}_3^{(2)}\)
are given by 
\begin{equation}
\label{eqn:TotalTower}
\mathrm{Gal}(\mathrm{k}_3^{(2)}/\mathrm{k})\simeq
\begin{cases}
\langle 729,9\rangle    & \text{ if } \tau_2(k)=\lbrack (3,3,3)^3;(3,3,3,3)\rbrack, \\
\langle 2187,123\rangle & \text{ if } \tau_2(k)=\lbrack (3,3,3,3)^3;(3,3,3,3,3)\rbrack, \\
\langle 2187,124\rangle & \text{ if } \tau_2(k)=\lbrack (9,3,3)^3;(3,3,3,3,3)\rbrack, \\
\langle 6561,i\rangle   & \text{ if } \tau_2(k)=\lbrack (3,3,3,3)^3;(3,3,3,3,3,3)\rbrack, \\
\langle 6561,109\rangle & \text{ if } \tau_2(k)=\lbrack (3,3,3,3)^3;(9,3,3,3,3)\rbrack, \\
\langle 6561,j\rangle   & \text{ if } \tau_2(k)=\lbrack (9,3,3)^3;(9,3,3,3,3)\rbrack,
\end{cases}
\end{equation}
with \(i\in\lbrace 103,105\rbrace\) and \(j\in\lbrace 110,111\rbrace\).
\end{theorem}


\begin{corollary}
\label{cor:TotalTower}
For the fields \(\mathrm{k}\)
with total capitulation \(\varkappa(\mathrm{k})=(000;4)\)
in Theorem
\ref{thm:TotalTower},
the length of the \(3\)-class field tower \(\mathrm{k}_3^{(\infty)}\) is given by
\begin{enumerate}
\item
\(\ell_3(\mathrm{k})\ge 2\), if \(G_2\in\lbrace\langle 729,9\rangle,\langle 2187,124\rangle,
\langle 6561,109\rangle,\langle 6561,110\rangle,\langle 6561,111\rangle\rbrace\),
\item
\(\ell_3(\mathrm{k})\ge 3\), if \(G_2\in\lbrace\langle 2187,123\rangle,\langle 6561,103\rangle,\langle 6561,105\rangle\rbrace\).
\end{enumerate}
\end{corollary}


\begin{proof}
(Proof of Theorem
\ref{thm:TotalTower}
and Corollary
\ref{cor:TotalTower}) 
Since the root \(\langle 729,9\rangle\) has nuclear rank \(\nu=3\),
it has descendants of step sizes \(s\in\lbrace 1,2,3\rbrace\).
The \(37\) children with \(s=3\) are of order \(3^9=19683\)
and possess abelian quotient invariants beyond the threshold \(\tau\ge\lbrack (9,9,3)^3;(3,3,3,3)\rbrack\).
Among the \(15\), resp. \(61\), children with \(s=1\), resp. \(s=2\), and order \(3^7=2187\), resp. \(3^8=6561\),
only the two, resp. five, mentioned possess an action by \(S_3\).
Concerning the length of \(3\)-class field towers,
the groups \(G_2\) in item (1) of the corollary have relation ranks \(4\le d_2\le 5\),
thus admitting a two-stage tower,
whereas those in item (2) have \(6\le d_2\le 7\),
which definitely excludes \(\ell_3(\mathrm{k})=2\).
\end{proof}

The following Figures
\ref{fig:C9xC3},
\ref{fig:Ord729Id17},
\ref{fig:Ord729Id20},
\ref{fig:Ord729Id18}
and
\ref{fig:Ord729Id21}
illustrate the location in descendant trees of
all metabelian groups \(M\) and certain non-metabelian groups \(G\)
which occur in section \S\
\ref{s:Tower}.


\section{Computational results}
\label{s:Computations}

\noindent
Table
\ref{tbl:Experiments}
has been computed with the aid of the computational algebra system MAGMA
\cite{MAGMA}.
For each of the $95$ pure metacyclic fields $\mathrm{k}=\mathbb{Q}(\sqrt[3]{p},\zeta_3)$
with prime radicands $p\equiv 1\,(\mathrm{mod}\,9)$ in the range $0<p<20\,000$
and $3$-class group of type $(9,3)$,
the capitulation kernels $\ker(T_{K_{i,3}/\mathrm{k}})$ of the class extension homomorphisms
$T_{K_{i,3}/\mathrm{k}}:\,C_{\mathrm{k},3}\to C_{K_{i,3},3}$
were computed and collected in the transfer kernel type $\varkappa$.

An asterisk indicates the second variant of harmonically balanced capitulation \(\varkappa=(123;4)\)
with abelian type invariants \(\tau=\lbrack(27,3)^3;(9,9,3)\rbrack\).


\begin{figure}[hb]
\caption{Finite \(3\)-groups \(G\) with commutator quotient \(G/G^\prime\simeq C_9\times C_3\)}
\label{fig:C9xC3}



{\tiny

\setlength{\unitlength}{0.9cm}
\begin{picture}(14,9)(-9,-8)

\put(-10,0.5){\makebox(0,0)[cb]{order}}

\put(-10,0){\line(0,-1){6}}
\multiput(-10.1,0)(0,-2){4}{\line(1,0){0.2}}

\put(-10.2,0){\makebox(0,0)[rc]{\(27\)}}
\put(-9.8,0){\makebox(0,0)[lc]{\(3^3\)}}
\put(-10.2,-2){\makebox(0,0)[rc]{\(81\)}}
\put(-9.8,-2){\makebox(0,0)[lc]{\(3^4\)}}
\put(-10.2,-4){\makebox(0,0)[rc]{\(243\)}}
\put(-9.8,-4){\makebox(0,0)[lc]{\(3^5\)}}
\put(-10.2,-6){\makebox(0,0)[rc]{\(729\)}}
\put(-9.8,-6){\makebox(0,0)[lc]{\(3^6\)}}

\put(-10,-6){\vector(0,-1){2}}

\put(-1.1,-0.1){\framebox(0.2,0.2){}}


\multiput(-9,-4)(0.5,0){4}{\circle{0.2}}
\put(-9,-6){\circle{0.2}}
\put(-8.5,-6){\circle{0.2}}
\multiput(-8,-6)(0.5,0){2}{\circle{0.2}}
\multiput(-7,-4)(0.5,0){4}{\circle{0.2}}

\put(-5,-2){\circle{0.2}}
\put(4,-2){\circle*{0.2}}
\put(4,-4){\circle{0.2}}
\put(5,-2){\circle{0.2}}

\multiput(-4,-6)(0.5,0){1}{\circle*{0.2}}
\multiput(-3.5,-6)(0.5,0){3}{\circle{0.2}}
\multiput(-1,-6)(0.5,0){2}{\circle{0.2}}
\multiput(0,-6)(0.5,0){4}{\circle*{0.2}}
\multiput(2,-6)(0.5,0){1}{\circle{0.2}}
\multiput(2.5,-6)(0.5,0){2}{\circle{0.2}}


\multiput(-9,-4)(0.5,0){4}{\line(0,-1){2}}
\put(-9,-6){\vector(0,-1){2}}
\multiput(-8,-6)(0.5,0){2}{\vector(0,-1){2}}

\put(-5,-2){\line(-2,-1){4}}
\put(-5,-2){\line(-1,-4){0.5}}
\put(-5,-2){\line(1,-4){1}}
\put(-5,-2){\line(2,-1){8}}

\put(-1,0){\line(-2,-1){4}}
\put(-1,0){\line(5,-2){5}}
\put(-1,0){\line(3,-1){6}}

\put(4,-2){\line(0,-1){2}}

\multiput(-4,-6)(0.5,0){4}{\vector(0,-1){2}}
\multiput(-1,-6)(0.5,0){7}{\vector(0,-1){2}}


\put(-1.1,0.1){\makebox(0,0)[rb]{\(\langle 2\rangle\)}}

\put(-9,-3.9){\makebox(0,0)[cb]{\(\mathrm{E}\)}}
\put(-8.5,-3.9){\makebox(0,0)[cb]{\(\mathrm{H}\)}}
\put(-8,-3.9){\makebox(0,0)[cb]{\(\mathrm{G}\)}}
\put(-7.5,-3.9){\makebox(0,0)[cb]{\(\mathrm{A}\)}}

\put(-9,-4.1){\makebox(0,0)[rt]{\(\langle 13\rangle\)}}
\put(-9,-6.1){\makebox(0,0)[rt]{\(\langle 65\rangle\)}}
\put(-8.5,-4.4){\makebox(0,0)[rt]{\(\langle 14\rangle\)}}
\put(-8.5,-6.4){\makebox(0,0)[rt]{\(\langle 73\rangle\)}}
\put(-8,-4.1){\makebox(0,0)[rt]{\(\langle 15\rangle\)}}
\put(-8,-6.1){\makebox(0,0)[rt]{\(\langle 79\rangle\)}}
\put(-7.5,-4.4){\makebox(0,0)[rt]{\(\langle 17\rangle\)}}
\put(-7.5,-6.4){\makebox(0,0)[rt]{\(\langle 84\rangle\)}}
\put(-7,-4.1){\makebox(0,0)[ct]{\(\langle 18\rangle\)}}
\put(-6.5,-4.4){\makebox(0,0)[ct]{\(\langle 16\rangle\)}}
\put(-6,-4.1){\makebox(0,0)[ct]{\(\langle 19\rangle\)}}
\put(-5.5,-4.4){\makebox(0,0)[ct]{\(\langle 20\rangle\)}}

\put(-5.1,-1.9){\makebox(0,0)[rb]{\(\langle 3\rangle\)}}
\put(3.9,-2.1){\makebox(0,0)[rt]{\(\langle 4\rangle\)}}
\put(4,-4.1){\makebox(0,0)[ct]{\(\langle 22\rangle\)}}
\put(4.5,-4.1){\makebox(0,0)[ct]{}}
\put(3.9,-6.1){\makebox(0,0)[rt]{}}
\put(4.5,-6.1){\makebox(0,0)[ct]{}}
\put(5,-2.1){\makebox(0,0)[ct]{\(\langle 6\rangle\)}}

\put(-4,-6.1){\makebox(0,0)[rt]{\(\langle 9\rangle\)}}
\put(-3.5,-6.4){\makebox(0,0)[rt]{\(\langle 10\rangle\)}}
\put(-3,-6.1){\makebox(0,0)[rt]{\(\langle 11\rangle\)}}
\put(-2.5,-6.4){\makebox(0,0)[rt]{\(\langle 12\rangle\)}}
\put(-1,-6.1){\makebox(0,0)[rt]{\(\langle 16\rangle\)}}
\put(-0.5,-6.4){\makebox(0,0)[rt]{\(\langle 19\rangle\)}}
\put(0,-6.1){\makebox(0,0)[rt]{\(\langle 17\rangle\)}}
\put(0.5,-6.4){\makebox(0,0)[rt]{\(\langle 20\rangle\)}}
\put(1,-6.1){\makebox(0,0)[rt]{\(\langle 18\rangle\)}}
\put(1.5,-6.4){\makebox(0,0)[rt]{\(\langle 21\rangle\)}}
\put(2,-6.1){\makebox(0,0)[rt]{\(\langle 13\rangle\)}}
\put(2.5,-6.4){\makebox(0,0)[ct]{\(\langle 14\rangle\)}}
\put(3,-6.1){\makebox(0,0)[ct]{\(\langle 15\rangle\)}}


\put(-1,-0.4){\makebox(0,0)[cb]{\(\mathrm{a}\)}}
\put(-1,-0.5){\makebox(0,0)[ct]{\(1\)}}

\put(-9,-8.4){\makebox(0,0)[cb]{\(\mathrm{b}\)}}
\put(-9,-8.5){\makebox(0,0)[ct]{\(15\)}}
\put(-8.5,-8.4){\makebox(0,0)[cb]{\(\mathrm{b}\)}}
\put(-8.5,-8.5){\makebox(0,0)[ct]{\(15\)}}
\put(-8,-8.4){\makebox(0,0)[cb]{\(\mathrm{a}\)}}
\put(-8,-8.5){\makebox(0,0)[ct]{\(1\)}}
\put(-7.5,-8.4){\makebox(0,0)[cb]{\(\mathrm{a}\)}}
\put(-7.5,-8.5){\makebox(0,0)[ct]{\(1\)}}

\put(-7,-5){\makebox(0,0)[cb]{\(\mathrm{b}\)}}
\put(-7,-5.1){\makebox(0,0)[ct]{\(16\)}}
\put(-6.5,-5){\makebox(0,0)[cb]{\(\mathrm{b}\)}}
\put(-6.5,-5.1){\makebox(0,0)[ct]{\(2\)}}
\put(-6,-5){\makebox(0,0)[cb]{\(\mathrm{b}\)}}
\put(-6,-5.1){\makebox(0,0)[ct]{\(3\)}}
\put(-5.5,-5){\makebox(0,0)[cb]{\(\mathrm{b}\)}}
\put(-5.5,-5.1){\makebox(0,0)[ct]{\(3\)}}

\put(-5,-2.7){\makebox(0,0)[cb]{\(\mathrm{a}\)}}
\put(-5,-2.8){\makebox(0,0)[ct]{\(1\)}}

\put(4,-4.7){\makebox(0,0)[cb]{\(\mathrm{A}\)}}
\put(4,-4.8){\makebox(0,0)[ct]{\(20\)}}

\put(5,-2.7){\makebox(0,0)[cb]{\(\mathrm{A}\)}}
\put(5,-2.8){\makebox(0,0)[ct]{\(1\)}}

\put(-4,-8.4){\makebox(0,0)[cb]{\(\mathrm{b}\)}}
\put(-4,-8.5){\makebox(0,0)[ct]{\(15\)}}
\put(-3.5,-8.4){\makebox(0,0)[cb]{\(\mathrm{b}\)}}
\put(-3.5,-8.5){\makebox(0,0)[ct]{\(31\)}}
\put(-3,-8.4){\makebox(0,0)[cb]{\(\mathrm{c}\)}}
\put(-3,-8.5){\makebox(0,0)[ct]{\(27\)}}
\put(-2.5,-8.4){\makebox(0,0)[cb]{\(\mathrm{A}\)}}
\put(-2.5,-8.5){\makebox(0,0)[ct]{\(20\)}}

\put(-1,-8.4){\makebox(0,0)[cb]{\(\mathrm{B}\)}}
\put(-1,-8.5){\makebox(0,0)[ct]{\(7\)}}
\put(-0.5,-8.4){\makebox(0,0)[cb]{\(\mathrm{B}\)}}
\put(-0.5,-8.5){\makebox(0,0)[ct]{\(7\)}}
\put(0,-8.4){\makebox(0,0)[cb]{\(\mathrm{E}\)}}
\put(0,-8.5){\makebox(0,0)[ct]{\(12\)}}
\put(0.5,-8.4){\makebox(0,0)[cb]{\(\mathrm{E}\)}}
\put(0.5,-8.5){\makebox(0,0)[ct]{\(12\)}}
\put(1,-8.4){\makebox(0,0)[cb]{\(\mathrm{e}\)}}
\put(1,-8.5){\makebox(0,0)[ct]{\(14\)}}
\put(1.5,-8.4){\makebox(0,0)[cb]{\(\mathrm{e}\)}}
\put(1.5,-8.5){\makebox(0,0)[ct]{\(14\)}}
\put(2,-8.4){\makebox(0,0)[cb]{\(\mathrm{d}\)}}
\put(2,-8.5){\makebox(0,0)[ct]{\(10\)}}
\put(2.5,-8.4){\makebox(0,0)[cb]{\(\mathrm{D}\)}}
\put(2.5,-8.5){\makebox(0,0)[ct]{\(11\)}}
\put(3,-8.4){\makebox(0,0)[cb]{\(\mathrm{D}\)}}
\put(3,-8.5){\makebox(0,0)[ct]{\(11\)}}

\put(-7.5,-10.1){\makebox(0,0)[lc]{Legend:}}
\put(-6,-10.1){\makebox(0,0)[lc]{\(\square\) \(\ldots\) abelian}}
\put(-6,-10.5){\makebox(0,0)[lc]{\(\bullet\) \(\ldots\) metabelian with required \(\varkappa\), \(\tau\) and action by \(S_3\)}}
\put(-6,-10.9){\makebox(0,0)[lc]{\(\circ\) \(\ldots\) other metabelian}}

\end{picture}

}

\end{figure}
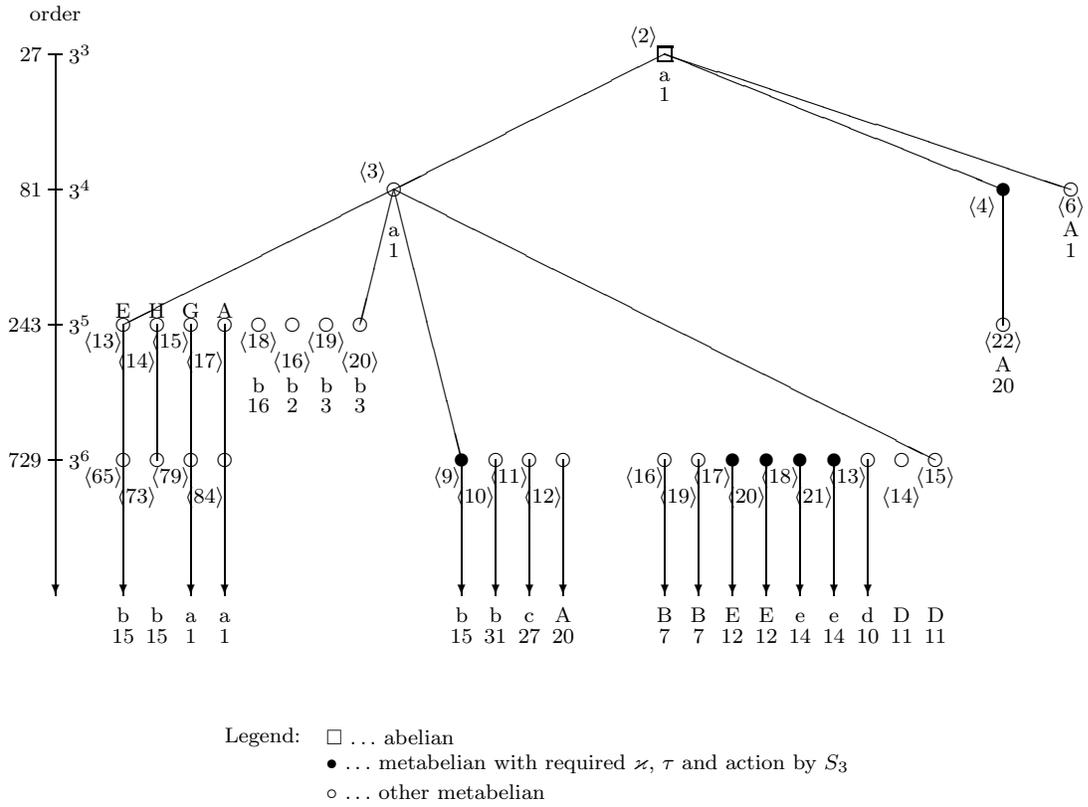

\newpage

\begin{figure}[ht]
\caption{Descendant tree of \(\langle 729,17\rangle\) with stable \(\varkappa=(123;4)\), \(\tau=\lbrack(27,3)^3;(9,3,3)\rbrack\)}
\label{fig:Ord729Id17}



{\tiny

\setlength{\unitlength}{0.8cm}
\begin{picture}(16,10)(-8,-15)

\put(-10,-5.5){\makebox(0,0)[cb]{order}}

\put(-10,-6){\line(0,-1){10}}
\multiput(-10.1,-6)(0,-2){6}{\line(1,0){0.2}}

\put(-10.2,-6){\makebox(0,0)[rc]{\(729\)}}
\put(-9.8,-6){\makebox(0,0)[lc]{\(3^6\)}}
\put(-10.2,-8){\makebox(0,0)[rc]{\(2\,187\)}}
\put(-9.8,-8){\makebox(0,0)[lc]{\(3^7\)}}
\put(-10.2,-10){\makebox(0,0)[rc]{\(6\,561\)}}
\put(-9.8,-10){\makebox(0,0)[lc]{\(3^8\)}}
\put(-10.2,-12){\makebox(0,0)[rc]{\(19\,683\)}}
\put(-9.8,-12){\makebox(0,0)[lc]{\(3^9\)}}
\put(-10.2,-14){\makebox(0,0)[rc]{\(59\,049\)}}
\put(-9.8,-14){\makebox(0,0)[lc]{\(3^{10}\)}}
\put(-10.2,-16){\makebox(0,0)[rc]{\(177\,147\)}}
\put(-9.8,-16){\makebox(0,0)[lc]{\(3^{11}\)}}

\put(-10,-16){\vector(0,-1){1}}

\put(-1.1,-5.9){\makebox(0,0)[lb]{\(\langle 17\rangle\)}}
\put(-1,-6){\circle*{0.2}}


\put(-5.6,-7.9){\makebox(0,0)[rb]{\(\langle 177\rangle\)}}
\put(-0.9,-7.9){\makebox(0,0)[lb]{\(\langle 178\rangle\)}}
\multiput(-5.5,-8)(4.5,0){2}{\circle*{0.2}}
\put(3.6,-7.9){\makebox(0,0)[lb]{\(\langle 179\rangle\)}}
\put(3.5,-8){\circle{0.2}}

\put(-6.1,-9.9){\makebox(0,0)[rb]{\(\langle 1731\rangle\)}}
\put(-6,-10){\circle*{0.2}}
\put(-4.9,-9.9){\makebox(0,0)[lb]{\(\langle 1732\rangle\)}}
\put(-5,-10){\circle{0.2}}

\put(-2.1,-9.9){\makebox(0,0)[rb]{\(\langle 1734\rangle\)}}
\put(-2,-10){\circle{0.2}}
\put(-1,-9.9){\makebox(0,0)[cb]{\(\langle 1733\rangle\)}}
\put(-1,-10){\circle*{0.2}}
\put(0.1,-9.9){\makebox(0,0)[lb]{\(\langle 1735\rangle\)}}
\put(-0.1,-10.1){\framebox(0.2,0.2){}}

\put(-6.6,-11.9){\makebox(0,0)[rb]{\(1\)}}
\put(-6.1,-11.9){\makebox(0,0)[rb]{\(2\)}}
\put(-5.6,-11.9){\makebox(0,0)[rb]{\(3\)}}
\multiput(-6.5,-12)(0.5,0){3}{\circle{0.2}}

\put(-3.1,-11.9){\makebox(0,0)[rb]{\(1\)}}
\put(-3,-12){\circle{0.2}}
\put(-2.1,-11.9){\makebox(0,0)[rb]{\(2\)}}
\put(-1.1,-11.9){\makebox(0,0)[rb]{\(3\)}}
\multiput(-2,-12)(1,0){2}{\circle*{0.2}}
\put(0.1,-11.9){\makebox(0,0)[lb]{\(4\)}}
\put(-0.1,-12.1){\framebox(0.2,0.2){\vrule height3pt width 3pt}}

\put(5.1,-11.8){\makebox(0,0)[lb]{\(1\)}}
\put(4.9,-12.1){\framebox(0.2,0.2){\vrule height3pt width 3pt}}
\put(7.1,-11.8){\makebox(0,0)[lb]{\(2\)}}
\put(6.9,-12.1){\framebox(0.2,0.2){}}

\multiput(-6.5,-14.2)(0.5,0){3}{\makebox(0,0)[ct]{\(1\)}}
\multiput(-6.6,-14.1)(0.5,0){3}{\framebox(0.2,0.2){}}

\multiput(-2,-14.2)(1,0){2}{\makebox(0,0)[ct]{\(1\)}}
\multiput(-2.1,-14.1)(1,0){2}{\framebox(0.2,0.2){}}

\put(4,-14.2){\makebox(0,0)[ct]{\(1\)}}
\put(3.9,-14.1){\framebox(0.2,0.2){}}
\put(5,-14.2){\makebox(0,0)[ct]{\(2\)}}
\put(6,-14.2){\makebox(0,0)[ct]{\(3\)}}
\multiput(4.9,-14.1)(1,0){2}{\framebox(0.2,0.2){\vrule height3pt width 3pt}}



\put(-1,-6){\line(-5,-2){4.5}}
\put(-1,-6){\line(0,-1){2}}
\put(-1,-6){\line(5,-2){4.5}}

\put(-5.5,-8){\line(-1,-4){0.5}}
\put(-5.5,-8){\line(1,-4){0.5}}

\put(-1,-8){\line(-1,-2){1}}
\put(-1,-8){\line(0,-1){2}}
\put(-1,-8){\line(1,-2){1}}

\put(-1,-8){\line(3,-2){6}}
\put(-1,-8){\line(2,-1){8}}

\put(-6,-10){\line(-1,-4){0.5}}
\put(-6,-10){\line(0,-1){2}}
\put(-6,-10){\line(1,-4){0.5}}

\put(-1,-10){\line(-1,-1){2}}
\put(-1,-10){\line(-1,-2){1}}
\put(-1,-10){\line(0,-1){2}}
\put(-1,-10){\line(1,-2){1}}

\multiput(-6.5,-12)(0.5,0){3}{\line(0,-1){2}}

\multiput(-2,-12)(1,0){2}{\line(0,-1){2}}

\put(5,-12){\line(-1,-2){1}}
\put(5,-12){\line(0,-1){2}}
\put(5,-12){\line(1,-2){1}}



\put(-7.5,-15.1){\makebox(0,0)[lc]{Legend:}}
\put(-6,-15.1){\makebox(0,0)[lc]{\(\bullet\) \(\ldots\) metabelian with action by \(S_3\)}}
\put(-6,-15.5){\makebox(0,0)[lc]{\(\circ\) \(\ldots\) other metabelian}}
\put(0,-15.2){\framebox(0.2,0.2){\vrule height3pt width 3pt}}
\put(0.3,-15.1){\makebox(0,0)[lc]{\(\ldots\) non-metabelian with action by \(S_3\)}}
\put(0,-15.5){\makebox(0,0)[lc]{\(\square\) \(\ldots\) other non-metabelian}}

\end{picture}

}

\end{figure}
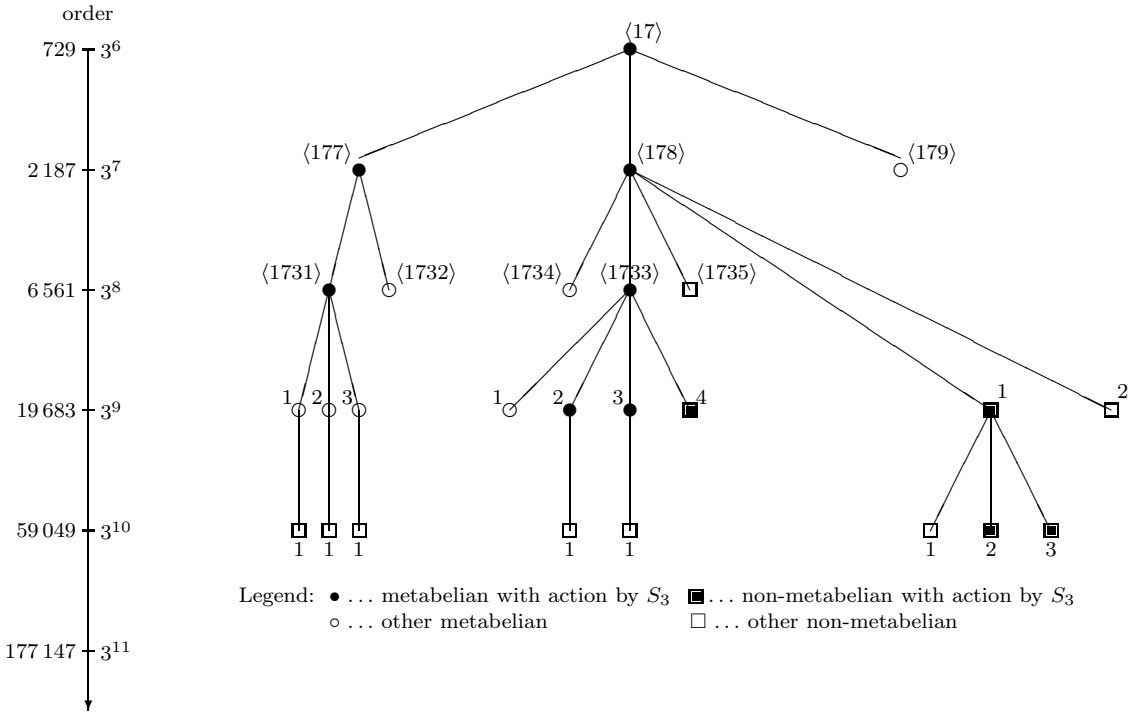

\vskip 0.5in

\begin{figure}[hb]
\caption{Descendant tree of \(\langle 729,20\rangle\) with stable \(\varkappa=(132;4)\), \(\tau=\lbrack(27,3)^3;(9,3,3)\rbrack\)}
\label{fig:Ord729Id20}



{\tiny

\setlength{\unitlength}{0.8cm}
\begin{picture}(16,10)(-8,-15)

\put(-10,-5.5){\makebox(0,0)[cb]{order}}

\put(-10,-6){\line(0,-1){10}}
\multiput(-10.1,-6)(0,-2){6}{\line(1,0){0.2}}

\put(-10.2,-6){\makebox(0,0)[rc]{\(729\)}}
\put(-9.8,-6){\makebox(0,0)[lc]{\(3^6\)}}
\put(-10.2,-8){\makebox(0,0)[rc]{\(2\,187\)}}
\put(-9.8,-8){\makebox(0,0)[lc]{\(3^7\)}}
\put(-10.2,-10){\makebox(0,0)[rc]{\(6\,561\)}}
\put(-9.8,-10){\makebox(0,0)[lc]{\(3^8\)}}
\put(-10.2,-12){\makebox(0,0)[rc]{\(19\,683\)}}
\put(-9.8,-12){\makebox(0,0)[lc]{\(3^9\)}}
\put(-10.2,-14){\makebox(0,0)[rc]{\(59\,049\)}}
\put(-9.8,-14){\makebox(0,0)[lc]{\(3^{10}\)}}
\put(-10.2,-16){\makebox(0,0)[rc]{\(177\,147\)}}
\put(-9.8,-16){\makebox(0,0)[lc]{\(3^{11}\)}}

\put(-10,-16){\vector(0,-1){1}}

\put(-1.1,-5.9){\makebox(0,0)[lb]{\(\langle 20\rangle\)}}
\put(-1,-6){\circle*{0.2}}


\put(-5.6,-7.9){\makebox(0,0)[rb]{\(\langle 187\rangle\)}}
\put(-0.9,-7.9){\makebox(0,0)[lb]{\(\langle 188\rangle\)}}
\multiput(-5.5,-8)(4.5,0){2}{\circle*{0.2}}
\put(3.6,-7.9){\makebox(0,0)[lb]{\(\langle 189\rangle\)}}
\put(3.5,-8){\circle{0.2}}

\put(-6.1,-9.9){\makebox(0,0)[rb]{\(\langle 1769\rangle\)}}
\put(-6,-10){\circle*{0.2}}
\put(-4.9,-9.9){\makebox(0,0)[lb]{\(\langle 1770\rangle\)}}
\put(-5,-10){\circle{0.2}}

\put(-2.1,-9.9){\makebox(0,0)[rb]{\(\langle 1772\rangle\)}}
\put(-2,-10){\circle{0.2}}
\put(-1,-9.9){\makebox(0,0)[cb]{\(\langle 1771\rangle\)}}
\put(-1,-10){\circle*{0.2}}
\put(0.1,-9.9){\makebox(0,0)[lb]{\(\langle 1773\rangle\)}}
\put(-0.1,-10.1){\framebox(0.2,0.2){}}

\put(-6.6,-11.9){\makebox(0,0)[rb]{\(1\)}}
\put(-6.1,-11.9){\makebox(0,0)[rb]{\(2\)}}
\put(-5.6,-11.9){\makebox(0,0)[rb]{\(3\)}}
\multiput(-6.5,-12)(0.5,0){3}{\circle{0.2}}

\put(-3.1,-11.9){\makebox(0,0)[rb]{\(1\)}}
\put(-3,-12){\circle{0.2}}
\put(-2.1,-11.9){\makebox(0,0)[rb]{\(2\)}}
\put(-1.1,-11.9){\makebox(0,0)[rb]{\(3\)}}
\multiput(-2,-12)(1,0){2}{\circle*{0.2}}
\put(0.1,-11.9){\makebox(0,0)[lb]{\(4\)}}
\put(-0.1,-12.1){\framebox(0.2,0.2){\vrule height3pt width 3pt}}

\put(5.1,-11.8){\makebox(0,0)[lb]{\(1\)}}
\put(4.9,-12.1){\framebox(0.2,0.2){\vrule height3pt width 3pt}}
\put(7.1,-11.8){\makebox(0,0)[lb]{\(2\)}}
\put(6.9,-12.1){\framebox(0.2,0.2){}}

\multiput(-6.5,-14.2)(0.5,0){3}{\makebox(0,0)[ct]{\(1\)}}
\multiput(-6.6,-14.1)(0.5,0){3}{\framebox(0.2,0.2){}}

\multiput(-2,-14.2)(1,0){2}{\makebox(0,0)[ct]{\(1\)}}
\multiput(-2.1,-14.1)(1,0){2}{\framebox(0.2,0.2){}}

\put(4,-14.2){\makebox(0,0)[ct]{\(1\)}}
\put(3.9,-14.1){\framebox(0.2,0.2){}}
\put(5,-14.2){\makebox(0,0)[ct]{\(2\)}}
\put(6,-14.2){\makebox(0,0)[ct]{\(3\)}}
\multiput(4.9,-14.1)(1,0){2}{\framebox(0.2,0.2){\vrule height3pt width 3pt}}



\put(-1,-6){\line(-5,-2){4.5}}
\put(-1,-6){\line(0,-1){2}}
\put(-1,-6){\line(5,-2){4.5}}

\put(-5.5,-8){\line(-1,-4){0.5}}
\put(-5.5,-8){\line(1,-4){0.5}}

\put(-1,-8){\line(-1,-2){1}}
\put(-1,-8){\line(0,-1){2}}
\put(-1,-8){\line(1,-2){1}}

\put(-1,-8){\line(3,-2){6}}
\put(-1,-8){\line(2,-1){8}}

\put(-6,-10){\line(-1,-4){0.5}}
\put(-6,-10){\line(0,-1){2}}
\put(-6,-10){\line(1,-4){0.5}}

\put(-1,-10){\line(-1,-1){2}}
\put(-1,-10){\line(-1,-2){1}}
\put(-1,-10){\line(0,-1){2}}
\put(-1,-10){\line(1,-2){1}}

\multiput(-6.5,-12)(0.5,0){3}{\line(0,-1){2}}

\multiput(-2,-12)(1,0){2}{\line(0,-1){2}}

\put(5,-12){\line(-1,-2){1}}
\put(5,-12){\line(0,-1){2}}
\put(5,-12){\line(1,-2){1}}



\put(-7.5,-15.1){\makebox(0,0)[lc]{Legend:}}
\put(-6,-15.1){\makebox(0,0)[lc]{\(\bullet\) \(\ldots\) metabelian with action by \(S_3\)}}
\put(-6,-15.5){\makebox(0,0)[lc]{\(\circ\) \(\ldots\) other metabelian}}
\put(0,-15.2){\framebox(0.2,0.2){\vrule height3pt width 3pt}}
\put(0.3,-15.1){\makebox(0,0)[lc]{\(\ldots\) non-metabelian with action by \(S_3\)}}
\put(0,-15.5){\makebox(0,0)[lc]{\(\square\) \(\ldots\) other non-metabelian}}

\end{picture}

}

\end{figure}
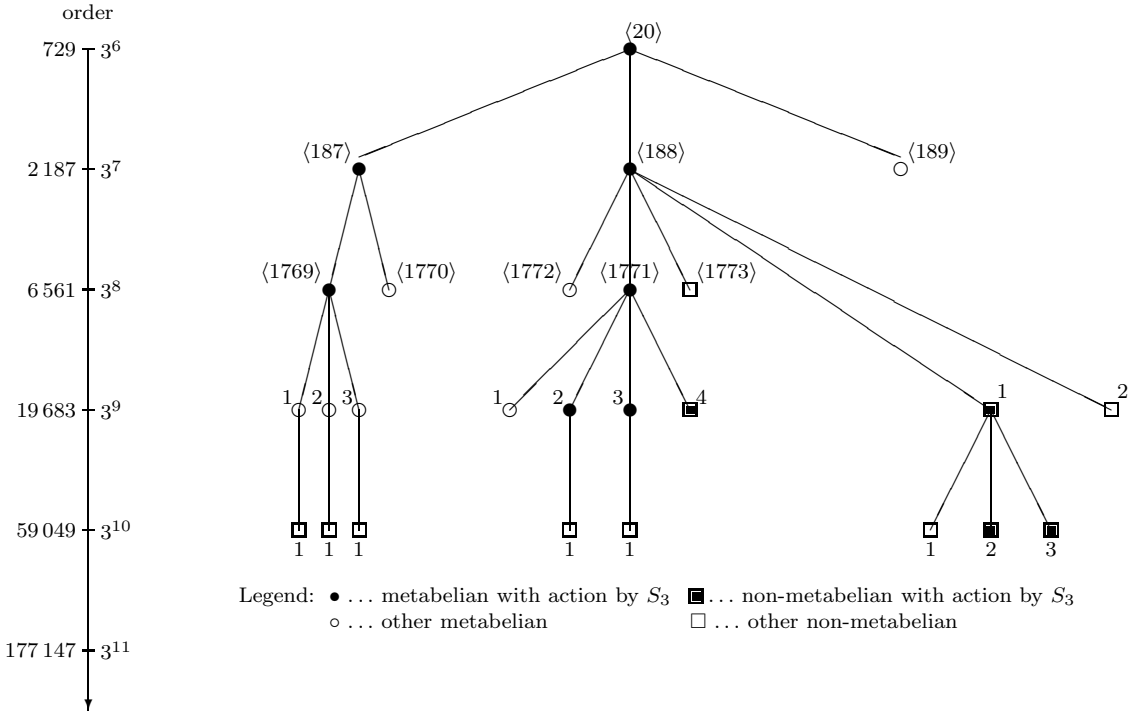

\newpage

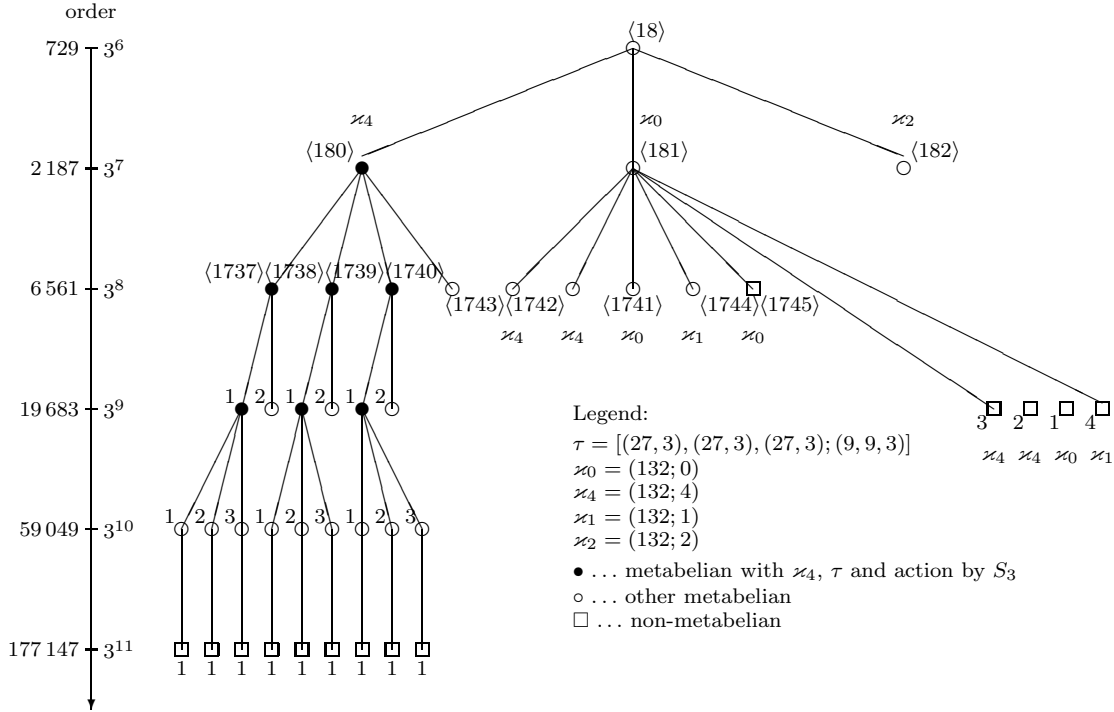
\begin{figure}[ht]
\caption{Descendant tree of \(\langle 729,18\rangle\)}
\label{fig:Ord729Id18}



{\tiny

\setlength{\unitlength}{0.8cm}
\begin{picture}(16,10)(-8,-15)


\put(-10,-5.5){\makebox(0,0)[cb]{order}}

\put(-10,-6){\line(0,-1){10}}
\multiput(-10.1,-6)(0,-2){6}{\line(1,0){0.2}}

\put(-10.2,-6){\makebox(0,0)[rc]{\(729\)}}
\put(-9.8,-6){\makebox(0,0)[lc]{\(3^6\)}}
\put(-10.2,-8){\makebox(0,0)[rc]{\(2\,187\)}}
\put(-9.8,-8){\makebox(0,0)[lc]{\(3^7\)}}
\put(-10.2,-10){\makebox(0,0)[rc]{\(6\,561\)}}
\put(-9.8,-10){\makebox(0,0)[lc]{\(3^8\)}}
\put(-10.2,-12){\makebox(0,0)[rc]{\(19\,683\)}}
\put(-9.8,-12){\makebox(0,0)[lc]{\(3^9\)}}
\put(-10.2,-14){\makebox(0,0)[rc]{\(59\,049\)}}
\put(-9.8,-14){\makebox(0,0)[lc]{\(3^{10}\)}}
\put(-10.2,-16){\makebox(0,0)[rc]{\(177\,147\)}}
\put(-9.8,-16){\makebox(0,0)[lc]{\(3^{11}\)}}

\put(-10,-16){\vector(0,-1){1}}


\put(-1.1,-5.9){\makebox(0,0)[lb]{\(\langle 18\rangle\)}}
\put(-1,-6){\circle{0.2}}


\put(-5.6,-7.9){\makebox(0,0)[rb]{\(\langle 180\rangle\)}}
\put(-5.5,-8){\circle*{0.2}}
\put(-0.9,-7.9){\makebox(0,0)[lb]{\(\langle 181\rangle\)}}
\put(3.6,-7.9){\makebox(0,0)[lb]{\(\langle 182\rangle\)}}
\multiput(-1,-8)(4.5,0){2}{\circle{0.2}}

\put(-7.1,-9.9){\makebox(0,0)[rb]{\(\langle 1737\rangle\)}}
\put(-6.1,-9.9){\makebox(0,0)[rb]{\(\langle 1738\rangle\)}}
\put(-5.1,-9.9){\makebox(0,0)[rb]{\(\langle 1739\rangle\)}}
\multiput(-7,-10)(1,0){3}{\circle*{0.2}}
\put(-4.1,-9.9){\makebox(0,0)[rb]{\(\langle 1740\rangle\)}}
\multiput(-4,-10)(1,0){1}{\circle{0.2}}
\put(-3.1,-10.1){\makebox(0,0)[rt]{\(\langle 1743\rangle\)}}
\put(-2.1,-10.1){\makebox(0,0)[rt]{\(\langle 1742\rangle\)}}
\put(-1,-10.1){\makebox(0,0)[ct]{\(\langle 1741\rangle\)}}
\put(0.1,-10.1){\makebox(0,0)[lt]{\(\langle 1744\rangle\)}}
\multiput(-3,-10)(1,0){4}{\circle{0.2}}
\put(1.1,-10.1){\makebox(0,0)[lt]{\(\langle 1745\rangle\)}}
\put(0.9,-10.1){\framebox(0.2,0.2){}}

\multiput(-7.6,-11.9)(1,0){3}{\makebox(0,0)[rb]{\(1\)}}
\multiput(-7.5,-12)(1,0){3}{\circle*{0.2}}
\multiput(-7.1,-11.9)(1,0){3}{\makebox(0,0)[rb]{\(2\)}}
\multiput(-7,-12)(1,0){3}{\circle{0.2}}
\put(4.9,-12.1){\makebox(0,0)[rt]{\(3\)}}
\put(5.5,-12.1){\makebox(0,0)[rt]{\(2\)}}
\put(6.1,-12.1){\makebox(0,0)[rt]{\(1\)}}
\put(6.7,-12.1){\makebox(0,0)[rt]{\(4\)}}
\multiput(4.9,-12.1)(0.6,0){4}{\framebox(0.2,0.2){}}

\multiput(-8.6,-13.9)(1.5,0){3}{\makebox(0,0)[rb]{\(1\)}}
\multiput(-8.1,-13.9)(1.5,0){3}{\makebox(0,0)[rb]{\(2\)}}
\multiput(-7.6,-13.9)(1.5,0){3}{\makebox(0,0)[rb]{\(3\)}}
\multiput(-8.5,-14)(0.5,0){9}{\circle{0.2}}

\multiput(-8.5,-16.2)(0.5,0){9}{\makebox(0,0)[ct]{\(1\)}}
\multiput(-8.6,-16.1)(0.5,0){9}{\framebox(0.2,0.2){}}



\put(-1,-6){\line(-5,-2){4.5}}
\put(-1,-6){\line(0,-1){2}}
\put(-1,-6){\line(5,-2){4.5}}

\put(-5.5,-7.2){\makebox(0,0)[cc]{\(\varkappa_4\)}}
\put(-0.9,-7.2){\makebox(0,0)[lc]{\(\varkappa_0\)}}
\put(3.5,-7.2){\makebox(0,0)[cc]{\(\varkappa_2\)}}

\put(-5.5,-8){\line(-3,-4){1.5}}
\put(-5.5,-8){\line(-1,-4){0.5}}
\put(-5.5,-8){\line(1,-4){0.5}}
\put(-5.5,-8){\line(3,-4){1.5}}

\put(-1,-8){\line(-1,-1){2}}
\put(-1,-8){\line(-1,-2){1}}
\put(-1,-8){\line(0,-1){2}}
\put(-1,-8){\line(1,-2){1}}
\put(-1,-8){\line(1,-1){2}}

\put(-3,-10.8){\makebox(0,0)[cc]{\(\varkappa_4\)}}
\put(-2,-10.8){\makebox(0,0)[cc]{\(\varkappa_4\)}}
\put(-1,-10.8){\makebox(0,0)[cc]{\(\varkappa_0\)}}
\put(0,-10.8){\makebox(0,0)[cc]{\(\varkappa_1\)}}
\put(1,-10.8){\makebox(0,0)[cc]{\(\varkappa_0\)}}

\put(-1,-8){\line(3,-2){6}}
\put(-1,-8){\line(2,-1){7.8}}

\put(5,-12.8){\makebox(0,0)[cc]{\(\varkappa_4\)}}
\put(5.6,-12.8){\makebox(0,0)[cc]{\(\varkappa_4\)}}
\put(6.2,-12.8){\makebox(0,0)[cc]{\(\varkappa_0\)}}
\put(6.8,-12.8){\makebox(0,0)[cc]{\(\varkappa_1\)}}

\multiput(-7,-10)(1,0){3}{\line(-1,-4){0.5}}
\multiput(-7,-10)(1,0){3}{\line(0,-1){2}}

\put(-7.5,-12){\line(-1,-2){1}}
\multiput(-7.5,-12)(1,0){2}{\line(-1,-4){0.5}}
\multiput(-7.5,-12)(1,0){3}{\line(0,-1){2}}
\multiput(-6.5,-12)(1,0){2}{\line(1,-4){0.5}}
\put(-5.5,-12){\line(1,-2){1}}

\multiput(-8.5,-14)(0.5,0){9}{\line(0,-1){2}}



\put(-2,-12.1){\makebox(0,0)[lc]{Legend:}}

\put(-2,-12.6){\makebox(0,0)[lc]{\(\tau=\lbrack(27,3),(27,3),(27,3);(9,9,3)\rbrack\)}}

\put(-2,-13){\makebox(0,0)[lc]{\(\varkappa_0=(132;0)\)}}
\put(-2,-13.4){\makebox(0,0)[lc]{\(\varkappa_4=(132;4)\)}}
\put(-2,-13.8){\makebox(0,0)[lc]{\(\varkappa_1=(132;1)\)}}
\put(-2,-14.2){\makebox(0,0)[lc]{\(\varkappa_2=(132;2)\)}}

\put(-2,-14.7){\makebox(0,0)[lc]{\(\bullet\) \(\ldots\) metabelian with \(\varkappa_4\), \(\tau\) and action by \(S_3\)}}
\put(-2,-15.1){\makebox(0,0)[lc]{\(\circ\) \(\ldots\) other metabelian}}
\put(-2,-15.5){\makebox(0,0)[lc]{\(\square\) \(\ldots\) non-metabelian}}

\end{picture}

}

\end{figure}

\vskip 0.5in

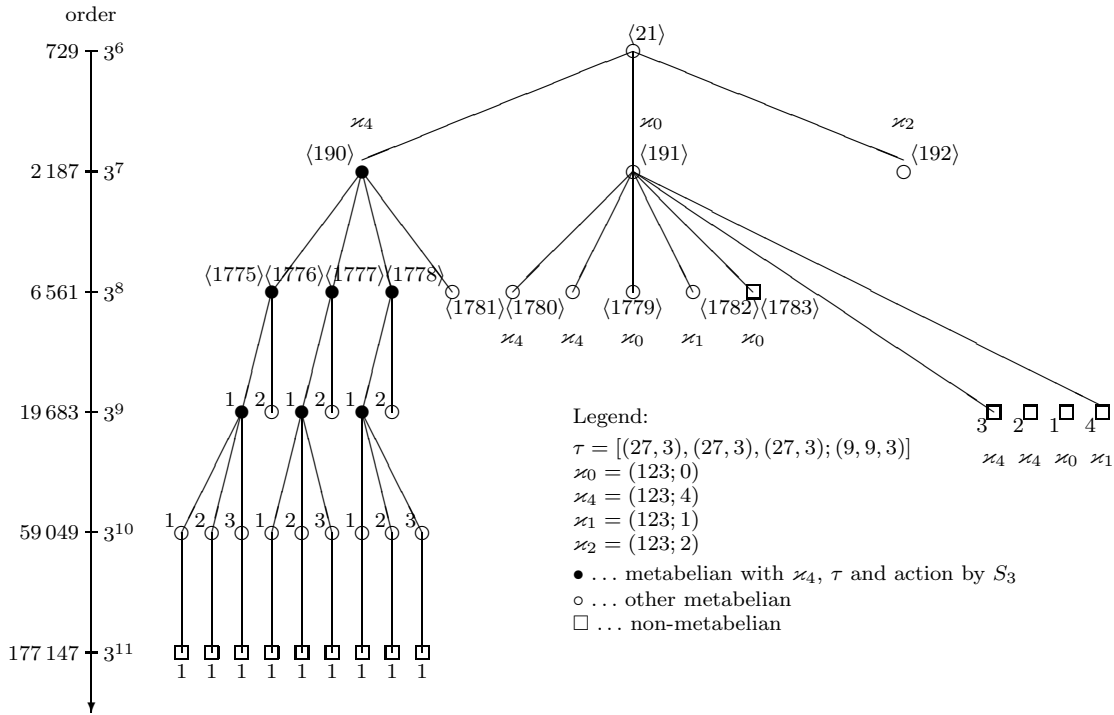
\begin{figure}[hb]
\caption{Descendant tree of \(\langle 729,21\rangle\)}
\label{fig:Ord729Id21}



{\tiny

\setlength{\unitlength}{0.8cm}
\begin{picture}(16,10)(-8,-15)


\put(-10,-5.5){\makebox(0,0)[cb]{order}}

\put(-10,-6){\line(0,-1){10}}
\multiput(-10.1,-6)(0,-2){6}{\line(1,0){0.2}}

\put(-10.2,-6){\makebox(0,0)[rc]{\(729\)}}
\put(-9.8,-6){\makebox(0,0)[lc]{\(3^6\)}}
\put(-10.2,-8){\makebox(0,0)[rc]{\(2\,187\)}}
\put(-9.8,-8){\makebox(0,0)[lc]{\(3^7\)}}
\put(-10.2,-10){\makebox(0,0)[rc]{\(6\,561\)}}
\put(-9.8,-10){\makebox(0,0)[lc]{\(3^8\)}}
\put(-10.2,-12){\makebox(0,0)[rc]{\(19\,683\)}}
\put(-9.8,-12){\makebox(0,0)[lc]{\(3^9\)}}
\put(-10.2,-14){\makebox(0,0)[rc]{\(59\,049\)}}
\put(-9.8,-14){\makebox(0,0)[lc]{\(3^{10}\)}}
\put(-10.2,-16){\makebox(0,0)[rc]{\(177\,147\)}}
\put(-9.8,-16){\makebox(0,0)[lc]{\(3^{11}\)}}

\put(-10,-16){\vector(0,-1){1}}


\put(-1.1,-5.9){\makebox(0,0)[lb]{\(\langle 21\rangle\)}}
\put(-1,-6){\circle{0.2}}


\put(-5.6,-7.9){\makebox(0,0)[rb]{\(\langle 190\rangle\)}}
\put(-5.5,-8){\circle*{0.2}}
\put(-0.9,-7.9){\makebox(0,0)[lb]{\(\langle 191\rangle\)}}
\put(3.6,-7.9){\makebox(0,0)[lb]{\(\langle 192\rangle\)}}
\multiput(-1,-8)(4.5,0){2}{\circle{0.2}}

\put(-7.1,-9.9){\makebox(0,0)[rb]{\(\langle 1775\rangle\)}}
\put(-6.1,-9.9){\makebox(0,0)[rb]{\(\langle 1776\rangle\)}}
\put(-5.1,-9.9){\makebox(0,0)[rb]{\(\langle 1777\rangle\)}}
\multiput(-7,-10)(1,0){3}{\circle*{0.2}}
\put(-4.1,-9.9){\makebox(0,0)[rb]{\(\langle 1778\rangle\)}}
\multiput(-4,-10)(1,0){1}{\circle{0.2}}
\put(-3.1,-10.1){\makebox(0,0)[rt]{\(\langle 1781\rangle\)}}
\put(-2.1,-10.1){\makebox(0,0)[rt]{\(\langle 1780\rangle\)}}
\put(-1,-10.1){\makebox(0,0)[ct]{\(\langle 1779\rangle\)}}
\put(0.1,-10.1){\makebox(0,0)[lt]{\(\langle 1782\rangle\)}}
\multiput(-3,-10)(1,0){4}{\circle{0.2}}
\put(1.1,-10.1){\makebox(0,0)[lt]{\(\langle 1783\rangle\)}}
\put(0.9,-10.1){\framebox(0.2,0.2){}}

\multiput(-7.6,-11.9)(1,0){3}{\makebox(0,0)[rb]{\(1\)}}
\multiput(-7.5,-12)(1,0){3}{\circle*{0.2}}
\multiput(-7.1,-11.9)(1,0){3}{\makebox(0,0)[rb]{\(2\)}}
\multiput(-7,-12)(1,0){3}{\circle{0.2}}
\put(4.9,-12.1){\makebox(0,0)[rt]{\(3\)}}
\put(5.5,-12.1){\makebox(0,0)[rt]{\(2\)}}
\put(6.1,-12.1){\makebox(0,0)[rt]{\(1\)}}
\put(6.7,-12.1){\makebox(0,0)[rt]{\(4\)}}
\multiput(4.9,-12.1)(0.6,0){4}{\framebox(0.2,0.2){}}

\multiput(-8.6,-13.9)(1.5,0){3}{\makebox(0,0)[rb]{\(1\)}}
\multiput(-8.1,-13.9)(1.5,0){3}{\makebox(0,0)[rb]{\(2\)}}
\multiput(-7.6,-13.9)(1.5,0){3}{\makebox(0,0)[rb]{\(3\)}}
\multiput(-8.5,-14)(0.5,0){9}{\circle{0.2}}

\multiput(-8.5,-16.2)(0.5,0){9}{\makebox(0,0)[ct]{\(1\)}}
\multiput(-8.6,-16.1)(0.5,0){9}{\framebox(0.2,0.2){}}



\put(-1,-6){\line(-5,-2){4.5}}
\put(-1,-6){\line(0,-1){2}}
\put(-1,-6){\line(5,-2){4.5}}

\put(-5.5,-7.2){\makebox(0,0)[cc]{\(\varkappa_4\)}}
\put(-0.9,-7.2){\makebox(0,0)[lc]{\(\varkappa_0\)}}
\put(3.5,-7.2){\makebox(0,0)[cc]{\(\varkappa_2\)}}

\put(-5.5,-8){\line(-3,-4){1.5}}
\put(-5.5,-8){\line(-1,-4){0.5}}
\put(-5.5,-8){\line(1,-4){0.5}}
\put(-5.5,-8){\line(3,-4){1.5}}

\put(-1,-8){\line(-1,-1){2}}
\put(-1,-8){\line(-1,-2){1}}
\put(-1,-8){\line(0,-1){2}}
\put(-1,-8){\line(1,-2){1}}
\put(-1,-8){\line(1,-1){2}}

\put(-3,-10.8){\makebox(0,0)[cc]{\(\varkappa_4\)}}
\put(-2,-10.8){\makebox(0,0)[cc]{\(\varkappa_4\)}}
\put(-1,-10.8){\makebox(0,0)[cc]{\(\varkappa_0\)}}
\put(0,-10.8){\makebox(0,0)[cc]{\(\varkappa_1\)}}
\put(1,-10.8){\makebox(0,0)[cc]{\(\varkappa_0\)}}

\put(-1,-8){\line(3,-2){6}}
\put(-1,-8){\line(2,-1){7.8}}

\put(5,-12.8){\makebox(0,0)[cc]{\(\varkappa_4\)}}
\put(5.6,-12.8){\makebox(0,0)[cc]{\(\varkappa_4\)}}
\put(6.2,-12.8){\makebox(0,0)[cc]{\(\varkappa_0\)}}
\put(6.8,-12.8){\makebox(0,0)[cc]{\(\varkappa_1\)}}

\multiput(-7,-10)(1,0){3}{\line(-1,-4){0.5}}
\multiput(-7,-10)(1,0){3}{\line(0,-1){2}}

\put(-7.5,-12){\line(-1,-2){1}}
\multiput(-7.5,-12)(1,0){2}{\line(-1,-4){0.5}}
\multiput(-7.5,-12)(1,0){3}{\line(0,-1){2}}
\multiput(-6.5,-12)(1,0){2}{\line(1,-4){0.5}}
\put(-5.5,-12){\line(1,-2){1}}

\multiput(-8.5,-14)(0.5,0){9}{\line(0,-1){2}}



\put(-2,-12.1){\makebox(0,0)[lc]{Legend:}}

\put(-2,-12.6){\makebox(0,0)[lc]{\(\tau=\lbrack(27,3),(27,3),(27,3);(9,9,3)\rbrack\)}}

\put(-2,-13){\makebox(0,0)[lc]{\(\varkappa_0=(123;0)\)}}
\put(-2,-13.4){\makebox(0,0)[lc]{\(\varkappa_4=(123;4)\)}}
\put(-2,-13.8){\makebox(0,0)[lc]{\(\varkappa_1=(123;1)\)}}
\put(-2,-14.2){\makebox(0,0)[lc]{\(\varkappa_2=(123;2)\)}}

\put(-2,-14.7){\makebox(0,0)[lc]{\(\bullet\) \(\ldots\) metabelian with \(\varkappa_4\), \(\tau\) and action by \(S_3\)}}
\put(-2,-15.1){\makebox(0,0)[lc]{\(\circ\) \(\ldots\) other metabelian}}
\put(-2,-15.5){\makebox(0,0)[lc]{\(\square\) \(\ldots\) non-metabelian}}

\end{picture}

}

\end{figure}

\newpage

The following distribution of capitulation types \(\varkappa\) arises in Table
\ref{tbl:Experiments}:
\begin{enumerate}
\item
\(61\) (\(64\%\)) with \(\varkappa=(444;4)\) (distinguished capitulation),
\item
\(14\) (\(15\%\)) with \(\varkappa=(123;4)\), \(\tau=\lbrack (27,3),(27,3),(27,3);(9,3,3)\rbrack\) (\(1^{\text{st}}\) variant),
\item
\(14\) (\(15\%\)) with \(\varkappa=(123;4*)\), \(\tau=\lbrack (27,3),(27,3),(27,3);(9,9,3)\rbrack\) (\(2^{\text{nd}}\) variant),
\item
\(6\) (\(6\%\)) with \(\varkappa=(000;4)\) (total capitulation).
\end{enumerate} 

\renewcommand{\arraystretch}{1.1}

\begin{table}[ht]
\caption{Capitulation types \(\varkappa\) of \(\mathrm{k}=\mathbb{Q}(\sqrt[3]{p},\zeta_3)\) in the range \(p<20\,000\)}
\label{tbl:Experiments}
\begin{center}
\begin{tabular}{|rrc||rrc||rrc||rrc|}
\hline
 No.  & \(p\) & \(\varkappa\) & No.  & \(p\) & \(\varkappa\) & No.   & \(p\) & \(\varkappa\) & No.   & \(p\) & \(\varkappa\) \\
\hline
  \(1\) &  \(199\) & \(4444\) & \(25\) & \(4951\) & \(4444\) & \(49\) &  \(9829\) & \(4444\) & \(73\) & \(14293\) & \(4444\) \\
  \(2\) &  \(271\) & \(4444\) & \(26\) & \(5059\) & \(4444\) & \(50\) & \(10243\) & \(4444\) & \(74\) & \(14419\) & \(4444\) \\
  \(3\) &  \(487\) & \(4444\) & \(27\) & \(5077\) & \(1234*\)& \(51\) & \(10459\) & \(0004\) & \(75\) & \(14563\) & \(4444\) \\
  \(4\) &  \(523\) & \(4444\) & \(28\) & \(5347\) & \(4444\) & \(52\) & \(10531\) & \(1234*\)& \(76\) & \(14779\) & \(4444\) \\
  \(5\) & \(1297\) & \(1234*\)& \(29\) & \(5437\) & \(1234\) & \(53\) & \(10657\) & \(4444\) & \(77\) & \(14923\) & \(4444\) \\
  \(6\) & \(1621\) & \(4444\) & \(30\) & \(5527\) & \(1234*\)& \(54\) & \(10837\) & \(0004\) & \(78\) & \(15121\) & \(0004\) \\
  \(7\) & \(1693\) & \(4444\) & \(31\) & \(5851\) & \(4444\) & \(55\) & \(10909\) & \(4444\) & \(79\) & \(15319\) & \(4444\) \\
  \(8\) & \(1747\) & \(1234\) & \(32\) & \(6067\) & \(4444\) & \(56\) & \(11251\) & \(4444\) & \(80\) & \(15427\) & \(4444\) \\
  \(9\) & \(1999\) & \(4444\) & \(33\) & \(6247\) & \(4444\) & \(57\) & \(11287\) & \(4444\) & \(81\) & \(16381\) & \(1234*\)\\
 \(10\) & \(2017\) & \(4444\) & \(34\) & \(6481\) & \(1234\) & \(58\) & \(11467\) & \(4444\) & \(82\) & \(16417\) & \(4444\) \\
 \(11\) & \(2143\) & \(1234*\)& \(35\) & \(6949\) & \(4444\) & \(59\) & \(11503\) & \(4444\) & \(83\) & \(16633\) & \(4444\) \\
 \(12\) & \(2377\) & \(4444\) & \(36\) & \(7219\) & \(0004\) & \(60\) & \(11593\) & \(4444\) & \(84\) & \(16993\) & \(1234*\)\\
 \(13\) & \(2467\) & \(1234\) & \(37\) & \(7507\) & \(1234*\)& \(61\) & \(11701\) & \(4444\) & \(85\) & \(17137\) & \(1234\) \\
 \(14\) & \(2593\) & \(1234\) & \(38\) & \(7687\) & \(4444\) & \(62\) & \(11719\) & \(4444\) & \(86\) & \(17209\) & \(1234*\)\\
 \(15\) & \(2917\) & \(4444\) & \(39\) & \(8011\) & \(1234\) & \(63\) & \(12097\) & \(4444\) & \(87\) & \(17497\) & \(1234*\)\\
 \(16\) & \(3511\) & \(4444\) & \(40\) & \(8209\) & \(4444\) & \(64\) & \(12511\) & \(1234\) & \(88\) & \(17569\) & \(4444\) \\
 \(17\) & \(3673\) & \(4444\) & \(41\) & \(8677\) & \(1234\) & \(65\) & \(12637\) & \(4444\) & \(89\) & \(18379\) & \(0004\) \\
 \(18\) & \(3727\) & \(4444\) & \(42\) & \(8821\) & \(4444\) & \(66\) & \(12853\) & \(4444\) & \(90\) & \(18451\) & \(1234\) \\
 \(19\) & \(3907\) & \(4444\) & \(43\) & \(9001\) & \(4444\) & \(67\) & \(12907\) & \(0004\) & \(91\) & \(18523\) & \(4444\) \\
 \(20\) & \(4159\) & \(4444\) & \(44\) & \(9109\) & \(4444\) & \(68\) & \(13159\) & \(4444\) & \(92\) & \(18541\) & \(4444\) \\
 \(21\) & \(4519\) & \(1234*\)& \(45\) & \(9343\) & \(1234\) & \(69\) & \(13177\) & \(4444\) & \(93\) & \(19387\) & \(1234\) \\
 \(22\) & \(4591\) & \(4444\) & \(46\) & \(9613\) & \(1234*\)& \(70\) & \(13339\) & \(4444\) & \(94\) & \(19441\) & \(1234*\)\\
 \(23\) & \(4789\) & \(1234\) & \(47\) & \(9631\) & \(4444\) & \(71\) & \(13411\) & \(4444\) & \(95\) & \(19927\) & \(1234*\)\\
 \(24\) & \(4933\) & \(4444\) & \(48\) & \(9721\) & \(4444\) & \(72\) & \(13807\) & \(1234\) &   \(\) &      \(\) &     \(\) \\
\hline
\end{tabular}
\end{center}
\end{table}


\section{Acknowledgements}
\label{s:Thanks}

\noindent
The third author gratefully acknowledges financial research subventions by the Austrian Science Fund (FWF): projects J0497-PHY and P26008-N25.



\begin{quote}
Siham AOUISSI \\
\'Ecole Normale Sup\'erieure, Department of Sciences, \\
Algebraic Theory and Application Research Group (ATA), \\
Moulay Ismail University, \\
B.P. 3104, Toulal, Mekn\`es - Morocco, \\
aouissi.siham@gmail.com.\\
\\
Mohamed TALBI \\
Regional Center of Professions of Education and Training, \\
60000 Oujda - Morocco, \\
ksirat1971@gmail.com. \\
\\
Daniel C. MAYER \\
Naglergasse 53, \\
8010 Graz - Austria, \\
quantum.algebra@icloud.com.\\
\\
Moulay Chrif ISMAILI \\
Department of Mathematics,
Mohammed First University, \\
60000 Oujda - Morocco, \\
mcismaili@yahoo.fr.  
\end{quote}


\end{document}